\documentclass[11pt]{amsart}    
\usepackage{latexsym, amsmath, amsthm, amssymb, setspace, verbatim}
\usepackage[all]{xy}
\usepackage{longtable}
\usepackage[utf8]{inputenc} \usepackage[T1]{fontenc} \usepackage{lmodern}
\usepackage{hyperref}
\usepackage{enumitem}

\title{ Sequentially split $*$-homomorphisms between \cstar-algebras }
\author{ Sel\c{c}uk Barlak \and G\'{a}bor Szab\'{o} }

\address{ Department of Mathematics and Computer Science, University of Southern \linebreak \text{}\hspace{3.0mm} Denmark, Campusvej 55, DK-5230 Odense M, Denmark }
\email{ barlak@imada.sdu.dk }
\address{ Fraser Noble Building, Institute of Mathematics, University of Aberdeen, \linebreak \text{}\hspace{3mm} Aberdeen AB24 3UE, Scotland, UK. }
\email{ gabor.szabo@abdn.ac.uk }
\thanks{ \emph{The first author was supported by:} SFB 878 \emph{Groups, Geometry and Actions}, GIF Grant 1137-30.6/2011, ERC AdG 267079 and the Villum Fonden project grant ‘Local and global structures of groups and their algebras’ (2014–2018)}
\thanks{ \emph{The second author was supported by:} SFB 878 \emph{Groups, Geometry and Actions} and EPSRC grant EP/N00874X/1. }

\numberwithin{equation}{section}

\begin{document}

\renewcommand\matrix[1]{\left(\begin{array}{*{10}{c}} #1 \end{array}\right)}  
\newcommand\set[1]{\left\{#1\right\}}  
\newcommand\mset[1]{\left\{\!\!\left\{#1\right\}\!\!\right\}}

\newcommand{\IA}[0]{\mathbb{A}} \newcommand{\IB}[0]{\mathbb{B}}
\newcommand{\IC}[0]{\mathbb{C}} \newcommand{\ID}[0]{\mathbb{D}}
\newcommand{\IE}[0]{\mathbb{E}} \newcommand{\IF}[0]{\mathbb{F}}
\newcommand{\IG}[0]{\mathbb{G}} \newcommand{\IH}[0]{\mathbb{H}}
\newcommand{\II}[0]{\mathbb{I}} \renewcommand{\IJ}[0]{\mathbb{J}}
\newcommand{\IK}[0]{\mathbb{K}} \newcommand{\IL}[0]{\mathbb{L}}
\newcommand{\IM}[0]{\mathbb{M}} \newcommand{\IN}[0]{\mathbb{N}}
\newcommand{\IO}[0]{\mathbb{O}} \newcommand{\IP}[0]{\mathbb{P}}
\newcommand{\IQ}[0]{\mathbb{Q}} \newcommand{\IR}[0]{\mathbb{R}}
\newcommand{\IS}[0]{\mathbb{S}} \newcommand{\IT}[0]{\mathbb{T}}
\newcommand{\IU}[0]{\mathbb{U}} \newcommand{\IV}[0]{\mathbb{V}}
\newcommand{\IW}[0]{\mathbb{W}} \newcommand{\IX}[0]{\mathbb{X}}
\newcommand{\IY}[0]{\mathbb{Y}} \newcommand{\IZ}[0]{\mathbb{Z}}

\newcommand{\CA}[0]{\mathcal{A}} \newcommand{\CB}[0]{\mathcal{B}}
\newcommand{\CC}[0]{\mathcal{C}} \newcommand{\CD}[0]{\mathcal{D}}
\newcommand{\CE}[0]{\mathcal{E}} \newcommand{\CF}[0]{\mathcal{F}}
\newcommand{\CG}[0]{\mathcal{G}} \newcommand{\CH}[0]{\mathcal{H}}
\newcommand{\CI}[0]{\mathcal{I}} \newcommand{\CJ}[0]{\mathcal{J}}
\newcommand{\CK}[0]{\mathcal{K}} \newcommand{\CL}[0]{\mathcal{L}}
\newcommand{\CM}[0]{\mathcal{M}} \newcommand{\CN}[0]{\mathcal{N}}
\newcommand{\CO}[0]{\mathcal{O}} \newcommand{\CP}[0]{\mathcal{P}}
\newcommand{\CQ}[0]{\mathcal{Q}} \newcommand{\CR}[0]{\mathcal{R}}
\newcommand{\CS}[0]{\mathcal{S}} \newcommand{\CT}[0]{\mathcal{T}}
\newcommand{\CU}[0]{\mathcal{U}} \newcommand{\CV}[0]{\mathcal{V}}
\newcommand{\CW}[0]{\mathcal{W}} \newcommand{\CX}[0]{\mathcal{X}}
\newcommand{\CY}[0]{\mathcal{Y}} \newcommand{\CZ}[0]{\mathcal{Z}}

\newcommand{\FA}[0]{\mathfrak{A}} \newcommand{\FB}[0]{\mathfrak{B}}
\newcommand{\FC}[0]{\mathfrak{C}} \newcommand{\FD}[0]{\mathfrak{D}}
\newcommand{\FE}[0]{\mathfrak{E}} \newcommand{\FF}[0]{\mathfrak{F}}
\newcommand{\FG}[0]{\mathfrak{G}} \newcommand{\FH}[0]{\mathfrak{H}}
\newcommand{\FI}[0]{\mathfrak{I}} \newcommand{\FJ}[0]{\mathfrak{J}}
\newcommand{\FK}[0]{\mathfrak{K}} \newcommand{\FL}[0]{\mathfrak{L}}
\newcommand{\FM}[0]{\mathfrak{M}} \newcommand{\FN}[0]{\mathfrak{N}}
\newcommand{\FO}[0]{\mathfrak{O}} \newcommand{\FP}[0]{\mathfrak{P}}
\newcommand{\FQ}[0]{\mathfrak{Q}} \newcommand{\FR}[0]{\mathfrak{R}}
\newcommand{\FS}[0]{\mathfrak{S}} \newcommand{\FT}[0]{\mathfrak{T}}
\newcommand{\FU}[0]{\mathfrak{U}} \newcommand{\FV}[0]{\mathfrak{V}}
\newcommand{\FW}[0]{\mathfrak{W}} \newcommand{\FX}[0]{\mathfrak{X}}
\newcommand{\FY}[0]{\mathfrak{Y}} \newcommand{\FZ}[0]{\mathfrak{Z}}

\newcommand{\Fa}[0]{\mathfrak{a}} \newcommand{\Fb}[0]{\mathfrak{b}}
\newcommand{\Fc}[0]{\mathfrak{c}} \newcommand{\Fd}[0]{\mathfrak{d}}
\newcommand{\Fe}[0]{\mathfrak{e}} \newcommand{\Ff}[0]{\mathfrak{f}}
\newcommand{\Fg}[0]{\mathfrak{g}} \newcommand{\Fh}[0]{\mathfrak{h}}
\newcommand{\Fi}[0]{\mathfrak{i}} \newcommand{\Fj}[0]{\mathfrak{j}}
\newcommand{\Fk}[0]{\mathfrak{k}} \newcommand{\Fl}[0]{\mathfrak{l}}
\newcommand{\Fm}[0]{\mathfrak{m}} \newcommand{\Fn}[0]{\mathfrak{n}}
\newcommand{\Fo}[0]{\mathfrak{o}} \newcommand{\Fp}[0]{\mathfrak{p}}
\newcommand{\Fq}[0]{\mathfrak{q}} \newcommand{\Fr}[0]{\mathfrak{r}}
\newcommand{\Fs}[0]{\mathfrak{s}} \newcommand{\Ft}[0]{\mathfrak{t}}
\newcommand{\Fu}[0]{\mathfrak{u}} \newcommand{\Fv}[0]{\mathfrak{v}}
\newcommand{\Fw}[0]{\mathfrak{w}} \newcommand{\Fx}[0]{\mathfrak{x}}
\newcommand{\Fy}[0]{\mathfrak{y}} \newcommand{\Fz}[0]{\mathfrak{z}}

\newcommand{\Ra}[0]{\Rightarrow}
\newcommand{\La}[0]{\Leftarrow}
\newcommand{\LRa}[0]{\Leftrightarrow}

\renewcommand{\phi}[0]{\varphi}
\newcommand{\eps}[0]{\varepsilon}

\newcommand{\quer}[0]{\overline}
\newcommand{\uber}[0]{\choose}
\newcommand{\ord}[0]{\operatorname{ord}}		
\newcommand{\GL}[0]{\operatorname{GL}}
\newcommand{\supp}[0]{\operatorname{supp}}	
\newcommand{\id}[0]{\operatorname{id}}		
\newcommand{\Sp}[0]{\operatorname{Sp}}		
\newcommand{\eins}[0]{\mathbf{1}}			
\newcommand{\diag}[0]{\operatorname{diag}}
\newcommand{\ind}[0]{\operatorname{ind}}
\newcommand{\auf}[1]{\quad\stackrel{#1}{\longrightarrow}\quad}
\newcommand{\hull}[0]{\operatorname{hull}}
\newcommand{\prim}[0]{\operatorname{Prim}}
\newcommand{\ad}[0]{\operatorname{Ad}}
\newcommand{\quot}[0]{\operatorname{Quot}}
\newcommand{\ext}[0]{\operatorname{Ext}}
\newcommand{\ev}[0]{\operatorname{ev}}
\newcommand{\fin}[0]{{\subset\!\!\!\subset}}
\newcommand{\diam}[0]{\operatorname{diam}}
\newcommand{\Hom}[0]{\operatorname{Hom}}
\newcommand{\Aut}[0]{\operatorname{Aut}}
\newcommand{\del}[0]{\partial}
\newcommand{\dimeins}[0]{\dim^{\!+1}}
\newcommand{\dimnuc}[0]{\dim_{\mathrm{nuc}}}
\newcommand{\dimnuceins}[0]{\dimnuc^{\!+1}}
\newcommand{\dr}[0]{\operatorname{dr}}
\newcommand{\dimrok}[0]{\dim_{\mathrm{Rok}}}
\newcommand{\dimrokeins}[0]{\dimrok^{\!+1}}
\newcommand{\dreins}[0]{\dr^{\!+1}}
\newcommand*\onto{\ensuremath{\joinrel\relbar\joinrel\twoheadrightarrow}} 
\newcommand*\into{\ensuremath{\lhook\joinrel\relbar\joinrel\rightarrow}}  
\newcommand{\im}[0]{\operatorname{im}}
\newcommand{\dst}[0]{\displaystyle}
\newcommand{\cstar}[0]{$\mathrm{C}^*$}
\newcommand{\ann}[0]{\operatorname{Ann}}
\newcommand{\dist}[0]{\operatorname{dist}}
\newcommand{\asdim}[0]{\operatorname{asdim}}
\newcommand{\asdimeins}[0]{\operatorname{asdim}^{\!+1}}
\newcommand{\amdim}[0]{\dim_{\mathrm{am}}}
\newcommand{\amdimeins}[0]{\amdim^{\!+1}}
\newcommand{\dimrokc}[0]{\dim_{\mathrm{Rok,c}}}
\newcommand{\dimrokceins}[0]{\dimrokc^{\!+1}}
\newcommand{\act}[0]{\operatorname{Act}}
\newcommand{\idlat}[0]{\operatorname{IdLat}}
\newcommand{\Cu}[0]{\operatorname{Cu}}
\newcommand{\Ost}[0]{\CO_\infty^{\mathrm{st}}}

\newtheorem{satz}{Satz}[section]		
\newtheorem{cor}[satz]{Corollary}
\newtheorem{lemma}[satz]{Lemma}
\newtheorem{prop}[satz]{Proposition}
\newtheorem{theorem}[satz]{Theorem}
\newtheorem*{theoreme}{Theorem}

\theoremstyle{definition}
\newtheorem{defi}[satz]{Definition}
\newtheorem*{defie}{Definition}
\newtheorem{defprop}[satz]{Definition \& Proposition}
\newtheorem{nota}[satz]{Notation}
\newtheorem*{notae}{Notation}
\newtheorem{rem}[satz]{Remark}
\newtheorem*{reme}{Remark}
\newtheorem{example}[satz]{Example}
\newtheorem{defnot}[satz]{Definition \& Notation}
\newtheorem{question}[satz]{Question}
\newtheorem*{questione}{Question}

\newenvironment{bew}{\begin{proof}[Proof]}{\end{proof}}

\begin{abstract}
We define and examine sequentially split $*$-homomorphisms between \cstar-algebras and \cstar-dynamical systems. For a $*$-homomorphism, the property of being sequentially split can be regarded as an approximate weakening of being a split-injective inclusion of \cstar-algebras. We show for a sequentially split $*$-homomorphism that a multitude of \cstar-algebraic approximation properties pass from the target algebra to the domain algebra, including virtually all important approximation properties currently used in the classification theory of \cstar-algebras. We also discuss various settings in which sequentially split $*$-homomorphisms arise naturally from context. One particular class of examples arises from compact group actions with the Rokhlin property. This allows us to recover and extend the presently known permanence properties of Rokhlin actions with a unified conceptual approach and a simple proof. Moreover, this perspective allows us to obtain new results about such actions, such as a generalization of Izumi's original $K$-theory formula for the fixed point algebra, or duality between the Rokhlin property and approximate representability.
\end{abstract}

\maketitle


\begin{samepage}
\tableofcontents
\end{samepage}

\setcounter{section}{-1}

\section{Introduction}
\noindent
In current \cstar-algebra theory, there is a substantial necessity to study interesting classes of examples of nuclear \cstar-algebras with the help of abstract regularity-type properties. The decisive motivation comes from the Elliott classification program for nuclear \cstar-algebras, and the fact that recent results in this area follow a rather abstract classification approach, in contrast to earlier results that presuppose some inductive limit decomposition of the \cstar-algebras under consideration. In this way, more satisfying and abstract classification results have emerged, largely depending on certain regularity properties.

On the one hand, these may include \emph{genuine} (i.e.~non-automatic) regularity properties such as finite decomposition rank \cite{KirchbergWinter2004}, finite nuclear dimension \cite{WinterZacharias2010}, $\CZ$-stability \cite{JiangSu1999, TomsWinter2007} or regularity of the Cuntz semigroup \cite[Section 3.2]{ElliottToms2008}. These properties play a key role in Toms-Winter's regularity conjecture, which asserts that these properties should be equivalent on a large scale \cite{ElliottToms2008, Winter2010, WinterZacharias2010}. On these properties alone, there exists a vast literature by now, and it would be impossible to do all the existing works justice within a mere paragraph of this introduction. Instead the reader is referred to \cite{BBSTWW} and the references therein, featuring an excellent overview of the current state-of-the-art in its introduction.

Another kind of regularity properties includes having finite tracial rank \cite{Lin2001}   or having generalized tracial rank at most one \cite{GongLinNiu2015}. The latter class of \cstar-algebras is an important one because they have been recently shown to be both classifiable and to exhaust the Elliott invariant.
Additionally, there are other kinds of more mysterious, yet important regularity-type properties for \cstar-algebras such as the universal coefficient theorem \cite{RosenbergSchochet1986, RosenbergSchochet1987}, where it is not known whether it is automatic for nuclear \cstar-algebras. In a very recent breakthrough by many hands \cite{GongLinNiu2015, ElliottNiu2015, ElliottGongLinNiu2015, TikuisisWhiteWinter2015}, it has been shown that these (a priori) different regularity-type properties are linked. This has culminated in the classification of separable, unital, simple \cstar-algebras with finite nuclear dimension that satisfy the UCT, see \cite[Section 6]{TikuisisWhiteWinter2015} for an overview.

If one wishes to show that a sample \cstar-algebra has one of the aforementioned \cstar-algebraic properties, then a frequent recipe for doing so is to find another \cstar-algebra, which is both easily seen to satisfy said property and has a certain connection to the original \cstar-algebra strong enough to carry it over. However, most of the proofs of this nature in the current literature follow a seemingly ad-hoc approach, and in particular it is often not clear what the crucial underlying common patterns really are. 
For instance, in order to study transformation group \cstar-algebras of the form $\CC(X)\rtimes\IZ$, one considers so-called orbit breaking subalgebras in the sense of Putnam \cite{Putnam1989}. Since these were shown by Lin-Phillips \cite{QLinPhillips1998} to be ASH algebras whose inductive limit decomposition incorporates subsets of $X$, these subalgebras are easier to understand than the crossed product \cstar-algebra itself. Nevertheless, the subalgebra turns out to be \emph{large} enough so that the crossed product inherits many of its properties. While this way of approach has been somewhat restricted only to a special crossed product setup for a long time, Phillips \cite{Phillips2014} has started to flesh out the general concept of a \emph{large subalgebra}. This does not only have the obvious advantage of opening up a lot of potential further possible ways of employing such an approach in other contexts, but can overall improve the understanding of the previous special cases of application.

This paper aims to serve a similar purpose by introducing sequentially split $*$-homomorphisms between \cstar-algebras. This property of a $*$-homomor\-phism can be understood as an approximate weakening of being a split-injective inclusion. In contrast to the large subalgebra setup, the existence of a sequentially split $*$-homomorphism from $A$ to $B$ guarantees that several approximation and perturbation properties pass from the larger algebra $B$ to the smaller algebra $A$. As it turns out, there are many places in the literature, not least in lengthy and technical calculations, where such a setup has implicitly occurred before. Our motivation is to unify such arguments with a conceptual approach. Throughout the entire paper, we will thus have a great focus on drawing a parallel between some of our (partial) results and the literature, for example by recasting known concepts in the language of sequentially split $*$-homomorphisms.

Let us now describe how this paper is organized. In the first preliminary section, we will recall some notions and techniques related to sequence algebras and central sequence algebras. The techniques developed in Kirchberg's pioneering work \cite{Kirchberg2004} on this subject are quintessential for the entire paper to enable proper treatment of non-unital \cstar-algebras. For the reader's convencience and because we need some of these results in an equivariant context later, we will also reprove or prove slight variants of some of the results from \cite{Kirchberg2004}, such as stability of the central sequence algebra.

In the second section, we introduce sequentially split $*$-homomorphisms between \cstar-algebras. We show that this notion behaves well with respect to standard constructions on separable \cstar-algebras, such as compositions, inductive limits, tensor products, or passing to hereditary subalgebras or quotients in a suitable manner. On the one hand, we will see that sequentially split $*$-homomorphisms are very rigid in the sense that they impose severe restrictions on the induced maps on \cstar-algebraic invariants, such as $K$-theory, traces or the Cuntz semigroup (cf.\ Theorem \ref{sequential dominance properties 1}). On the other hand, the main result of the second section (cf.\ Theorem \ref{sequential dominance properties 2}) asserts that if there exists any sequentially split $*$-homomorphism from a \cstar-algebra $A$ to another \cstar-algebra $B$, than a host of \cstar-algebraic properties of interest pass from $B$ to $A$. This includes many approximation properties such as regularity properties in the Toms-Winter conjecture, (generalized) tracial rank at most one, or being expressible as an inductive limit of certain weakly semiprojective \cstar-algebras, such as 1-NCCW complexes. Somewhat surprisingly, the property of being nuclear and satisfying the UCT also turns out to pass from $B$ to $A$ (cf.\ Theorem \ref{UCT}). This is a further manifestation of an observation by Dadarlat \cite{Dadarlat2003} that nuclearity together with the UCT can be understood as an approximation or perturbation property of some kind. To summarize, the abridged version of our combined main result in the second section reads as follows:

\begin{theoreme}
Let $A$ and $B$ be two \cstar-algebras. Suppose that there exists a sequentially split $*$-homomorphism from $A$ to $B$. If $B$ is classifiable in the sense of the Elliott program, then so is $A$.
\end{theoreme}

In the third section, we consider sequentially split $*$-homomorphisms in the equivariant context, i.e.~as a property of an equivariant $*$-homomorphism between \cstar-dynamical systems. This notion behaves well with respect to taking crossed products: if one has an equivariantly sequentially split $*$-homomorphism between two \cstar-dynamical systems $(A,\alpha, G)$ and $(B,\beta,G)$, then the induced $*$-homomorphism between $A\rtimes_\alpha G$ and $B\rtimes_\beta G$ is also sequentially split. In the case that the acting group $G$ is abelian, the induced map between the crossed products is $\hat{G}$-equivariant (with respect to the dual actions), and is then also equivariantly sequentially split. Using Takai-duality, we will see that the converse holds as well (cf.\ Theorem \ref{seq split duality}). That is, a $*$-homomorphism $\phi: (A,\alpha,G)\to (B,\beta,G)$ is equivariantly sequentially split if and only if the dual morphism $\hat{\phi}: (A\rtimes_\alpha G, \hat{\alpha},\hat{G})\to (B\rtimes_\beta G, \hat{\beta},\hat{G})$ is equivariantly sequentially split.

In the fourth section, we will discuss several applications that we shall now summarize:

First, we will see that a Rokhlin action $\alpha: G\curvearrowright A$ of a compact group always gives rise to sequentially split $*$-homomorphisms $A^\alpha\to A$ and $A\rtimes_\alpha G\to A\otimes\CK(L^2(G))$ in a natural way (cf.\ Theorem \ref{Rokhlin fixed point seq split}). We note that a similar observation was made by Gardella in \cite{Gardella2014_R}, and also by the second author in \cite{Szabo2015} in the context of the continuous Rokhlin property. Therefore, the concept of sequentially split $*$-homomorphisms fleshes out many of the arguments related to permanence properties appearing in the literature of Rokhlin actions of either finite groups or compact groups \cite{OsakaPhillips2012, Santiago2014, HirshbergWinter2007, Gardella2014_R}. Exploiting this viewpoint for Rokhlin actions some more, we prove a $K$-theory formula for the fixed point algebra of a Rokhlin action (cf.\ Theorem \ref{Rokhlin K-theory}), generalizing such a $K$-theory formula for finite groups established by Izumi in \cite{Izumi2004}. Another one of Izumi's observations from \cite{Izumi2004} is that for finite abelian groups, Rokhlin actions are in a natural way dual to approximately representable actions. Motivated by this, we extend Izumi-Matui's definition of approximately representable actions in \cite{IzumiMatui2010} to the setting of discrete group actions on not necessarily unital \cstar-algebras. It turns out that, similar to the Rokhlin property, approximate representability can be characterized in terms of (equivariantly) sequentially split $*$-homomorphisms. Armed with this perspective, we then generalize Izumi's duality result and prove that Rokhlin actions of compact abelian groups are in a natural way dual to approximately representable actions of discrete abelian groups (cf.\ Theorem \ref{duality classical case}). Note that a similar observation was made by Gardella \cite{Gardella2014} for circle actions on unital \cstar-algebras; see also \cite{GardellaPHD} for a further generalization in the unital case.

We then consider Osaka-Kodaka-Teruya's notion of an inclusion of unital \cstar-algebras $A\subset B$ with the Rokhlin property. This was introduced in \cite{OsakaKodakaTeruya2009} and studied further in \cite{OsakaTeruya2011, OsakaTeruya2014}, motivated by the fact that in this setup many interesting \cstar-algebraic properties pass from $B$ to $A$. We show that if an inclusion $A\subset B$ has the Rokhlin property in the sense of \cite{OsakaKodakaTeruya2009}, then the inclusion map is sequentially split (cf.\ Theorem \ref{Rokhlin inclusions}). In particular, the main results of this paper recover and extend the previously known permanence properties for inclusions of unital \cstar-algebras with the Rokhlin property.

We continue in the fourth section by showing that for separable \cstar-algebras, sequentially split $*$-homomorphisms can also be understood as a generalization of an existential embedding (cf.\ Theorem \ref{ec-seq-split}), as considered by Goldbring-Sinclair in \cite[Section 2]{GoldbringSinclair2015}. We note that since the initial preprint version of this paper was published, the strong connection between sequantially split $*$-homomorphisms and the model theory of \cstar-algebras and \cstar-dynamical systems has been pinned down in \cite{GoldbringSinclair16, GardellaLupini16}. In particular, it has since been discovered that the notion of sequentially split $*$-homomorphisms, restricted to separable \cstar-algebras, agrees with the notion of so-called positively existential embeddings. 

We end the fourth section by considering \cstar-dynamical systems absorbing a given strongly self-absorbing action tensorially in the sense of \cite{Szabo15ssa}. As it turns out, this can also be expressed in the language of sequentially split $*$-homomorphisms. Exploiting some of our general observations for equivariantly sequentially split $*$-homomorphisms, we prove a few permanence properties for \cstar-dynamical systems absorbing a given strongly self-absorbing action tensorially: this property turns out to be invariant under equivariant Morita equivalence, and moreover passes to a system that comes from an invariant, hereditary subalgebra.

We are confident that the concept of sequentially split $*$-homomorphisms has enough worth to be fleshed out, and is certain to find some possible new applications in the future. Since the initial preprint version of this paper was available, the recent papers \cite{GoldbringSinclair16, GardellaLupini16} have emerged as evidence toward this claim. In work of the first author joint with Omland-Stammeier \cite{BarlakOmlandStammeier16}, the main results of this paper are applied for the purpose of $K$-theoretic computations. Moreover, building on the techniques presented in this paper, the authors have developed a theory for Rokhlin actions of compact quantum groups in collaboration with Christian Voigt; see \cite{BarlakSzaboVoigt2016}. We would like to thank him for some fruitful discussion on some of the topics of this paper. For the same reason, the authors are grateful to Ilijas Farah and Isaac Goldbring. We also thank Martino Lupini for pointing out an error in the first preprint version of this paper.
 We would like to express our gratitude to the referee for some useful suggestions.


\section{Preliminaries}

\noindent
Let us first fix some notations:

\begin{notae}
\begin{enumerate}[label={$\bullet$},leftmargin=*]
\item If $F$ is a finite subset inside some larger set $M$, we write $F\fin M$.
\item Let $\eps>0$ and $a, b$ some elements in a normed space. We write $a=_\eps b$ as a shortcut for $\|a-b\|\leq\eps$.
\item Let $\eps>0$ and let $M_1, M_2$ be subsets of some normed space. If the distance from $M_1$ to $M_2$ is at most $\eps$, then we write $M_1\subset_\eps M_2$.
\end{enumerate}
\end{notae}

Let us now recall some notions related to sequence algebras and central sequence algebras. The techniques developed in Kirchberg's pioneering work \cite{Kirchberg2004} on this subject are quintessential for what follows:

\begin{nota}[cf.\ {\cite[1.1]{Kirchberg2004}}] \label{basic constructions}
Let $A$ be a \cstar-algebra and $B\subset A$ a \cstar-subalgebra. We denote by
\[
A\cap B' = \set{x\in A ~|~ xb=bx~\text{for all}~b\in B}
\]
the relative commutant of $B$ inside $A$. Consider the two-sided annihilator of $B$ inside $A$
\[
\ann(B,A) = \set{ x\in A ~|~ xb=bx=0~\text{for all}~b\in B}.
\]
Consider also the normalizer of $B$ inside $A$
\[
\CN(B,A) = \set{ x\in A ~|~ xB+Bx \subset B}.
\]
There is a chain of inclusions
\[
\ann(B,A)\subset A\cap B'\subset \CN(\overline{BAB},A),
\]
where $\ann(B,A)$ is an ideal in both of these \cstar-algebras.
This allows one to define
\[
F(B,A) = A\cap B'/\ann(B,A).
\]
\end{nota}

\begin{nota}
Let $A$ be a \cstar-algebra and $\omega$ a free filter on $\IN$. Recall the ($\omega$-)sequence algebra $A_\omega$ of $A$ given by
\[
A_\omega = \ell^\infty(\IN, A)/c_\omega(\IN,A) = \ell^\infty(\IN,A)/\set{ (a_n)_n\in\ell^\infty(\IN,A) ~|~ \lim_{n\to\omega} \|a_n\|=0 }.
\]
Given a bounded sequence $(a_n)_n\in \ell^\infty(\IN, A)$, the norm of the corresponding element in $A_\omega$ is given by
\[
\| [(a_n)_n] \| = \limsup_{n\to \omega} \|a_n \|=\inf_{J\in \omega}\sup_{n\in J} \|a_n\|.
\]
Then $A$ embeds canonically into $A_\omega$ as (representatives of) constant sequences. We will frequently use this identification of $A$ inside $A_\omega$ without mention. 
\end{nota}

\begin{rem}[cf.~{\cite[1.1, 1.9(2)+(3), 1.10(2)]{Kirchberg2004}}] 
\label{reminder-kirchberg}
If $B\subset A_\omega$ is a \cstar-subalgebra, then the constructions from \ref{basic constructions} apply. In this context, we additionally set $D_{B,A_\omega} = \quer{B\cdot A_\omega\cdot B}$ as a shortcut. If $B$ is $\sigma$-unital, then the existence of a countable approximate unit in $B$, together with a reindexation argument, allows one to find a positive contraction $e\in A_\omega$ with $eb=be=b$ for all $b\in B$. Then $e\in A_\omega\cap B'$ and in fact, its class $e+\ann(B,A_\omega)\in F(B,A_\omega)$ is a unit. Thus we see that $F(B,A_\omega)$ is unital, if $B$ is $\sigma$-unital.

In the special case that we consider the standard inclusion $A\subset A_\omega$, we write $D_{\omega,A}=D_{A,A_\omega}$ and $F_\omega(A)=F(A,A_\omega)$. 
In particular, we see that if $A$ is a $\sigma$-unital \cstar-algebra, then its central sequence algebra $F_\omega(A)$ is unital. However, if $A$ is in fact unital, then we see that $\ann(A,A_\omega)=0$ and so $F_\omega(A)$ simply recovers the ordinary central sequence algebra $A_\omega\cap A'$.
\end{rem}

\begin{rem}[cf.~{\cite[1.1]{Kirchberg2004}}] \label{multiplication FA}
It is clear from the definitions that
\[
F(B,A_\omega)\otimes_{\operatorname{max}} B \to D_{B,A_\omega}\subset A_\omega,\ (a+\ann(B,A_\omega))\otimes b \mapsto ab,
\]
is a well-defined, natural $*$-homomorphism. This proves to be important in applications.
\end{rem}

The following was proved by Kirchberg in \cite[1.9]{Kirchberg2004} for free ultrafilters on $\IN$. We give a more direct proof for arbitrary free filters on $\IN$, not making use of the full power of the $\eps$-test as in \cite{Kirchberg2004}. 

\begin{prop}[cf.~{\cite[1.9(4)+(5)]{Kirchberg2004}}] 
\label{surjection normalizer multiplier}
Let $A$ be a \cstar-algebra and $\omega$ a free filter on $\IN$. Let $B$ be a $\sigma$-unital \cstar-subalgebra of $A_\omega$. 
\begin{enumerate}[label=\textup{(\arabic*)}]
\item The canonical $*$-homomorphism $\CN(B,A_\omega)\to \CM(B)$ given by the universal property of the multiplier algebra is surjective and its kernel coincides with $\ann(B,A_\omega)$. \label{snm:1}
\item The $*$-isomorphism $\CN(D_{B,A_\omega},A_\omega)/\ann(B,A_\omega)\stackrel{\cong}{\longrightarrow} \CM(D_{B,A_\omega})$ obtained from \ref{snm:1} restricts to a $*$-isomorphism from $F(B,A_\omega)$ onto $\CM(D_{B,A_\omega})\cap B'$. \label{snm:2}
\end{enumerate}
\end{prop}
\begin{proof}
\ref{snm:1}: Since $B\subset \CN(B,A_\omega)$ is an ideal, there is a unique $*$-homomorphism $\pi: \CN(B,A_\omega) \to \CM(B)$ extending the identity map on $B$. More explicitly, for $m\in \CN(B,A_\omega)$, $\pi(m)$ is given by
\[
\pi(m)b=mb\quad \text{and} \quad b\pi(m)=bm\quad \text{for}\ b\in B.
\]
Moreover,
\[
\ker(\pi) = \set{x\in \CN(B,A_\omega) ~|~  xb=bx=0~\text{for all}~b\in B} = \ann(B,A_\omega).
\]

Let $h\in B$ be a strictly positive element. Let $\rho_\omega:A_\infty \to A_\omega$ denote the  natural surjection. Take a positive lift $k\in A_\infty$ for $h$ and define 
\[
C:=\overline{k\rho_\omega^{-1}(B)k}\subset A_\infty.
\]
Then $C$ is a $\sigma$-unital \cstar-subalgebra of $A_\infty$ satisfying $\rho_\omega(C)=B$. The surjection $\rho_\omega:C\to B$ extends to a strictly continuous $*$-epimorphism $\CM(C)\to \CM(B)$. Observe that $\rho_\omega$ maps $\CN(C,A_\infty)$ into $\CN(B,A_\omega)$. Consider also the canonical $*$-homomorphism $\pi': \CN(C,A_\infty)\to\CM(C)$. We obtain a commutative diagram
\[
\xymatrix{
 \CN(C,A_\infty) \ar[r]^{\rho_\omega} \ar[d]^{\pi'} & \CN(B,A_\omega) \ar[d]^{\pi} \\
 \CM(C) \ar@{->>}[r] & \CM(B)
}
\]
So in order to prove the assertion, we may without loss of generality restrict to the case $\omega=\infty$.

Let $M\in \CM(B)$. For each $n\in \IN$, we find some $k\in \IN$ such that $a_n=h^{1/k}Mh^{1/k}\in B$ satisfies
\[
\|a_nh - Mh \|, \|ha_n - hM \|<\frac{1}{n}.
\]
Let us now represent all these elements by bounded sequences, i.e.~write $h=[(h_k)_k]$, $a_n=[(a^{(n)}_k)_k]$, $Mh=[(b_k)_k]$ and $hM=[(c_k)_k]$, respectively. We find a strictly increasing sequence of natural numbers $(n_k)_k$ with the property that
\[
\|a^{(j)}_k h_k - b_k \|,\|h_k a^{(j)}_k - c_k\| \leq \frac{1}{j}\quad \text{for all}\ j\in \IN \text{ and } k\geq n_j.
\]
Define $m\in A_\infty$ as the equivalence class of the bounded sequence given by
\[
m_k:=
\begin{cases}
0 &,\quad \text{if}\ 0\leq k < n_1,\\
a^{(j)}_k &,\quad \text{if}\ n_j\leq k < n_{j+1}.
\end{cases}
\]
By construction,
\[
\limsup_{k\to \infty}\| m_k h_k - b_k \| \leq \inf_{j\in \IN} \sup_{k\geq n_j} \| m_k h_k - b_k \| \leq \inf_{j\in \IN} \frac{1}{j}=0.
\]
Analogously, we also obtain that $\limsup_{k\to \infty}\| h_k m_k - c_k \|=0$. Hence, $mh=Mh$ and $hm=hM$. As $h$ is strictly positive for $B$ and $mh,hm\in B$, one concludes that $m\in \CN(B,A_\infty)$. As the multipliers $M$ and $\pi(m)$ coincide on the strictly positive element $h\in B$, we get that $\pi(m)=M$. This shows that $\pi$ is surjective.

\ref{snm:2}: First note that $\ann(B,A_\omega)=\ann(D_{B,A_\omega},A_\omega)$, since the inclusion $B\subset D_{B,A_\omega}$ is non-degenerate. Therefore, \ref{snm:1} indeed induces an isomorphism $\CN(D_{B,A_\omega},A_\omega)/\ann(B,A_\omega)\stackrel{\cong}{\longrightarrow} \CM(D_{B,A_\omega})$. Since $A_\omega\cap B'\subset\CN(D_{B,A_\omega},A_\omega)$, this isomorphism restricts to an inclusion from $F(B,A_\omega)$ into $\CM(D_{B,A_\omega})\cap B'$. In order to show that it is surjective, it suffices to show that
\[
A_\omega\cap B' = \pi^{-1}\big( \CM(D_{B,A_\omega})\cap B' ).
\]
If $y\in \CM(D_{B,A_\omega})\cap B'$, consider any lift $x\in \CN(D_{B,A_\omega},A_\omega)$ with $\pi(x)=y$. Since $y$ lies in the relative commutant of $B$, we know that $[b,x]\in\ann(B,A_\omega)$ for all $b\in B$ by construction of $\pi$. But then $xb=xb^{1/2}b^{1/2}=b^{1/2}xb^{1/2}=bx$ for all $b\in B_+$, using this fact twice for the square root $b^{1/2}$. Thus $x\in A_\omega\cap B'$, which finishes the proof.
\end{proof}

\begin{prop} \label{D with compacts}
Let $A$ be a \cstar-algebra, $\CK$ the compact operators on some separable Hilbert space and $\omega$ a free filter on $\IN$. Then
the canonical embedding from $A_\omega\otimes\CK$ into $(A\otimes\CK)_\omega$ restricts to an isomorphism from $D_{\omega,A}\otimes\CK$ onto $D_{\omega, A\otimes\CK}$.
\end{prop}
\begin{proof}
Note that in the case of a finite-dimensional Hilbert space, this is obvious. In that case, for each $n\in\IN$, it is well-known that $A_\omega\otimes M_n\cong (A\otimes M_n)_\omega$ via the canonical embedding, and thus
\[
\begin{array}{ccl}
D_{\omega,A}\otimes M_n &=& \quer{AA_\omega A}\otimes M_n \\\\
&=& \quer{(A\otimes \eins_n) (A_\omega\otimes M_n) (A\otimes \eins_n)} \\\\
&\cong& \quer{(A\otimes M_n) (A\otimes M_n)_\omega (A\otimes M_n)} \\\\
&=& D_{\omega, A\otimes M_n}.
\end{array}
\]
Now let $\CK$ be the compacts on the separable and infinite-dimensional Hilbert space. If we view $M_n$ embedded into $M_{n+1}$ as the upper left corner, then it follows that
\[
\begin{array}{ccl}
D_{\omega, A}\otimes\CK &=& \quer{\bigcup_{n\in\IN} D_{\omega,A}\otimes M_n} \\\\
&\cong& \quer{\bigcup_{n\in\IN} D_{\omega, A\otimes M_n}} \\\\
&=& \quer{\bigcup_{n\in\IN} (A\otimes M_n)(A\otimes M_n)_\omega(A\otimes M_n)} \\\\
&=& \quer{\bigcup_{n\in\IN} (A\otimes M_n)(A\otimes\CK)_\omega(A\otimes M_n)} \\\\
&=& \quer{(A\otimes\CK)(A\otimes\CK)_\omega(A\otimes\CK)} = D_{\omega, A\otimes\CK}.
\end{array}
\]
This finishes the proof.
\end{proof}

The following is a well-known fact among \cstar-algebraists. However, we will include a proof for the reader's convenience, as the authors had trouble finding a reference where this is made explicit.

\begin{prop} \label{relative commutant compacts general}
Let $A$ be a \cstar-algebra and let $\CK$ denote the \cstar-algebra of compact operators on some Hilbert space. Then
\[
\CM(A \otimes \CK) \cap (\eins \otimes \CK)' = \CM(A) \otimes \eins.
\]
\end{prop}
\begin{proof}
Fix a generating set of matrix units $\set{e_{kl}\ |\ k,l\in I}$ for $\CK$, where $I$ is an index set whose cardinality matches the dimension of the underlying Hilbert space. Let $m \in \CM(A \otimes \CK) \cap (\eins \otimes \CK)'$ and $(u_\lambda)_\lambda$ an approximate unit for $A$. Using $m(\eins \otimes e_{ij}) = (\eins \otimes e_{ij})m $, we get
\[
m(u_\lambda \otimes e_{ii}) =  m(u_\lambda \otimes e_{ij}e_{jj}e_{ji}) = (\eins \otimes e_{ij}) m (u_\lambda \otimes e_{jj})(\eins \otimes e_{ji})
\]
for $i,j\in I$. If $i=j$, this shows that $m(u_\lambda \otimes e_{ii}) = a_{\lambda,i} \otimes e_{ii}$ for some $a_{\lambda,i}\in A$. The computation now shows that for all $i,j \in I$,
\[
a_{\lambda,i} \otimes e_{ii} = (\eins \otimes e_{ij}) (a_{\lambda,j} \otimes e_{jj}) (\eins \otimes e_{ji}) = a_{\lambda,j} \otimes e_{ii}.
\]
This implies $a_{\lambda,i} = a_{\lambda,j} = a_\lambda$ for all $i,j \in I$ and all $\lambda$. For $L\fin I$ and each $\lambda$, we get
\[
m \cdot \Bigl(\sum\limits_{i\in L} u_\lambda \otimes e_{ii}\Bigl) = a_\lambda \otimes \sum\limits_{i\in L} e_{ii}.
\]
As the net $\dst \Bigl(\sum\limits_{i\in L} u_\lambda \otimes e_{ii} \Bigl)_{(\lambda,L)}$, with the obvious underlying directed set, is an approximate unit for $A\otimes \CK$, we see that the left hand side of the equation converges strictly to $m \in \CM(A\otimes \CK)$. Hence, the net $\dst \Bigl( a_\lambda \otimes \sum\limits_{i\in L} e_{ii} \Bigl)_{(\lambda,L)}$ also converges strictly to $m \in \CM(A\otimes \CK)$. Let $F\fin A$ be a finite subset and $\eps > 0$. Fix some $k \in I$. Then for sufficiently large $(\lambda_1,L_1), (\lambda_2,L_2)$ with $k \in L_1 \cap L_2$ and every $a'\in F$, we obtain
\[
\begin{array}{ccl}
\| (a_{\lambda_1} - a_{\lambda_2})a' \| &=& \| ((a_{\lambda_1} - a_{\lambda_2})a')\otimes e_{kk} \| \\\\
&=& \dst \Bigl\| \Bigl( a_{\lambda_1} \otimes \bigl( \sum\limits_{i\in L_1} e_{ii} \bigl) - a_{\lambda_2} \otimes \bigl( \sum\limits_{i\in L_2} e_{ii} \bigl) \Bigl) a' \otimes e_{kk} \Bigl\| \leq \eps.
\end{array}
\]
and analogously $\| a'(a_{\lambda_1} - a_{\lambda_2}) \| \leq \eps$. This shows that $(a_\lambda)_\lambda$ converges strictly to some $c \in \CM(A)$. It follows that $\dst \Bigl( a_\lambda \otimes \sum\limits_{i=1}^n e_{ii} \Bigl)_{(\lambda,L)}$ converges strictly to $c\otimes \eins$. This shows that $m = c\otimes \eins \in \CM(A)\otimes \eins$. As the other inclusion is trivial, the proof is complete.
\end{proof}

\begin{cor} \label{relative commutant compact}
Let $A$ be a \cstar-algebra, $\CK$ the compact operators on some separable Hilbert space and $\omega$ a free filter on $\IN$. Then
\[
\CM(D_{\omega,A\otimes \CK})\cap (\eins\otimes \CK)'=\CM(D_{\omega,A})\otimes \eins.
\]
\end{cor}
\begin{proof}
By \ref{D with compacts}, $\CM(D_{\omega,A\otimes \CK}) \cong \CM(D_{\omega,A}\otimes \CK)$ naturally. In particular, the canonical subalgebras $\eins \otimes \CK$ on both sides get identified under this isomorphism. The claim now follows directly from \ref{relative commutant compacts general}.
\end{proof}

Next, we use this observation to prove that $F_\omega(A)$ is a stable invariant for $\sigma$-unital \cstar-algebras. For free ultrafilters this is already known due to Kirchberg's pioneering work on central sequences of \cstar-algebras \cite{Kirchberg2004}.

\begin{prop}[cf.~{\cite[1.10(3)]{Kirchberg2004}}] 
\label{morita invariance FA}
Let $A$ be a $\sigma$-unital \cstar-algebra, $\CK$ the compact operators on some separable Hilbert space and $\omega$ a free filter on $\IN$. Then the canonical $*$-monomorphism
\[
F_\omega(A) \into F_\omega(A\otimes \CK),\ a\mapsto a\otimes \eins,
\]
given by the identifications of \ref{surjection normalizer multiplier}\ref{snm:2}
is an isomorphism.
\end{prop}
\begin{proof}
By \ref{surjection normalizer multiplier}, it is enough to show that the inclusion
\[
\CM(D_{\omega,A})\cap A'\into \CM(D_{\omega,A\otimes \CK})\cap (A\otimes \CK)',
\]
which is induced by the first-factor embedding, is surjective. We have
\[
\begin{array}{cll}
\CM(D_{\omega,A\otimes \CK})\cap (A\otimes \CK)' &=& (\CM(D_{\omega, A\otimes \CK}) \cap (\eins\otimes \CK)')\cap (A\otimes \eins)' \\
&\stackrel{\ref{relative commutant compact}}{=}& (\CM(D_{\omega,A}) \cap A')\otimes \eins,
\end{array}
\]
and hence this embedding is indeed onto.
\end{proof}

We will also make use of the following standard fact:

\begin{lemma}[cf.\ {\cite[15.2.2]{Loring1997} and \cite[3.9]{Thiel2011}}] \label{perturb weakly semiproj}
Let $\FC$ be a class of separable, weakly semiprojective \cstar-algebras. Let $A$ be a separable \cstar-algebra that can be locally approximated by \cstar-algebras in $\FC$. Then $A$ is an inductive limit of \cstar-algebras in $\FC$.
\end{lemma}


\section{Sequentially split homomorphisms: The non-equivariant case}

\begin{defi} Let $A$ and $B$ be two \cstar-algebras. A $*$-homomorphism $\phi:A\to B$ is called sequentially split, if there exists a commutative diagram of $*$-homomorphisms of the form
\[
\xymatrix{
A \ar[dr]_\phi \ar[rr] && A_\infty \\
& B \ar[ur]_\psi &
}
\]
where the horizontal map is the canonical inclusion. If $\psi: B\to A_\infty$ is a $*$-homomorphism fitting into the above diagram, then we say that $\psi$ is an approximate left-inverse for $\phi$.

We say that $A$ is sequentially dominated by $B$ or that $B$ sequentially dominates $A$, if there exists a sequentially split $*$-homomorphism $\phi:A \to B$.
\end{defi}

\begin{lemma} \label{weak-seq-split}
Let $A$ and $B$ be two separable \cstar-algebras. A $*$-homomorphism $\phi:A\to B$ is sequentially split if and only if there exists a $*$-homomorphism $\psi:B\to (A_\infty)_\infty$ such that the following diagram commutes 
\[
\xymatrix{
A \ar[dr]_\phi \ar[rr] && (A_\infty)_\infty \\
& B \ar[ur]_\psi &
}
\]
\end{lemma}
\begin{proof}
Evidently, we only have to prove the ``if''-part. Find $*$-linear maps $\psi_{(m,n)}:B\to A$, $m,n \in \IN$, such that $(\psi_{(m,n)})_{(m,n)}: B\to\ell^\infty(\IN^2,A)$ is well-defined and lifts $\psi$. As $\psi$ is a $*$-homomorphism, we have
\[
\limsup\limits_{m \to \infty}~\limsup\limits_{n\to\infty} \| \psi_{(m,n)}(bb')-\psi_{(m,n)}(b)\psi_{(m,n)}(b') \| = 0
\]
and
\[
\limsup\limits_{m \to \infty}~\limsup\limits_{n \to \infty} \| \psi_{(m,n)}(b) \| \leq \| b \|
\]
for all $b,b'\in B$.
Moreover, we have
\[
\limsup_{m \to \infty}~\limsup_{n\to\infty} \| \psi_{(m,n)}\circ\phi(a)-a\| = 0
\]
for all $a\in A$ by choice of $\psi$.

Let $S_1\subset S_2\subset\dots\fin A$ be an increasing sequence of finite subsets $S_k$ such that $A' = \bigcup_{k\in \IN} S_k$ is dense in $A$. Let $F_1\subset F_2\subset\dots\fin B$ be an increasing sequence of finite subsets $F_k$ such that $B' = \bigcup_{k \in \IN} F_k$ is a dense $\IQ[i]$-$*$-subalgebra of $B$. We may assume that $\phi(A') \subset B'$. 

By the above, there is a sequence of natural numbers $(m_k)_{k \in \IN}$ such that for all $k \in \IN$, $b,b' \in F_k$ and $a \in S_k$,
\[
\limsup\limits_{n \to \infty}\| \psi_{(m_k,n)}(bb') - \psi_{(m_k,n)}(b)\psi_{(m_k,n)}(b') \| \leq \frac{1}{k}
\]
and
\[
\limsup\limits_{n \to \infty} \| \psi_{(m_k,n)}(b) \| \leq \| b \| + \frac{1}{k}
\]
and
\[
\limsup_{n \to \infty} \| \psi_{(m_k,n)}\circ\phi(a)-a\|\leq \frac{1}{k}.
\]
Now we find a sequence of natural numbers $(n_k)_k$ such that for all $k \in \IN$, $b,b' \in F_k$ and $a \in S_k$, we have
\[
\| \psi_{(m_k,n_k)}(bb') - \psi_{(m_k,n_k)}(b)\psi_{(m_k,n_k)}(b') \| \leq \frac{2}{k}
\]
and
\[
 \| \psi_{(m_k,n_k)}(b) \| \leq \| b \| + \frac{2}{k}
\]
and
\[
\| \psi_{(m_k,n_k)}\circ\phi(a)-a\|\leq \frac{2}{k}.
\]
The resulting $*$-linear map $(\psi_{(m_k,n_k)})_k : B\to \ell^\infty(\IN,A)$ is well-defined because for $x\in B$, we have $\sup_{k\in\IN} \|\psi_{(m_k,n_k)}(x)\| \leq \sup_{m,n\in\IN} \|\psi_{(m,n)}(x)\| < \infty$.

It now follows directly from the construction that the $\IQ[i]$-$*$-linear map
\[
\psi':B'\to A_\infty,\ \psi'(b)=[(\psi_{(m_k,n_k)}(b))_k],
\]
is a contractive $*$-homomorphism satisfying $\psi'\circ \phi(a) = a$ for all $a \in A'$. Hence, $\psi'$ extends uniquely to a $*$-homomorphism $\psi':B \to A_\infty$. Moreover, as $A' \subset A$ is dense, $\psi' \circ \phi$ coincides with the canonical embedding of $A$ into $A_\infty$. This shows that $\phi$ is sequentially split.
\end{proof}

\begin{prop} 
\label{seq-split-comp}
Restricted to separable \cstar-algebras, the composition of any two sequentially split $*$-homomorphisms is sequentially split.
\end{prop}
\begin{proof}
Let $A,B$ and $C$ be separable \cstar-algebras and assume that
$\varphi:A\to B$ and $\psi:B\to C$ are sequentially split $*$-homomorphisms. We obtain a commutative diagram
\[
	\xymatrix{
	A \ar[rr] \ar[rd]_\varphi && A_\infty \ar[rr] && (A_\infty)_\infty	\\
	& B \ar[rd]_\psi \ar[ru] \ar[rr] && B_\infty \ar[ru] &\\
	&& C \ar[ru] &&
	}
\]
where the horizontal maps are the respective standard embeddings. It now follows from \ref{weak-seq-split} that $\psi\circ\varphi: A \to C$ is sequentially split.
\end{proof}

\begin{lemma}
\label{unital homs in FA}
Let $A$ be a $\sigma$-unital \cstar-algebra and $C$ a unital \cstar-algebra. There exists a unital $*$-homomorphism from $C$ to $F_\infty(A)$ if and only if the first-factor embedding $\id_A\otimes\eins: A\to A\otimes_{\max} C$ is sequentially split.
\end{lemma}
\begin{proof}
The ``only if'' part is clear because for every such $C$, the $*$-homomor\-phism from \ref{multiplication FA} provides an approximate left-inverse for the first-factor embedding.

So let us show the ``if'' part. Assuming that $\id_A\otimes\eins: A\to A\otimes_{\max} C$ is sequentially split, let $\psi: A\otimes_{\max} C\to A_\infty$ be an approximate left-inverse. Since the first-factor embedding is non-degenerate, the image of $\psi$ must be contained in $D_{\infty,A}$, and the resulting $*$-homomorphism $\psi: A\otimes_{\max} C\to D_{\infty,A}$ is also non-degenerate. Let us consider the unique strictly continuous extension $\psi: \CM(A\otimes_{\max} C)\to\CM(D_{\infty,A})$. Since $\psi(a\otimes\eins)=a$ for every $a\in A$, it follows that $\psi(\eins\otimes C)\subset\CM(D_{\infty,A})\cap A'$. By \ref{surjection normalizer multiplier}\ref{snm:2}, the right-hand side is naturally isomorphic to $F_\infty(A)$, so this yields the existence of a unital $*$-homomorphism from $C$ to $F_\infty(A)$.
\end{proof}

\begin{prop} 
\label{unital homs in FA seq split}
Let $A$ be a separable \cstar-algebra and $B$ a $\sigma$-unital \cstar-algebra. Assume that $\phi:A\to B$ is a sequentially split $*$-homomorphism. Let $C$ be a unital, separable \cstar-algebra and assume that there exists a unital $*$-homomorphism from $C$ to $F_\infty(B)$. Then there exists a unital $*$-homomorphism from $C$ to $F_\infty(A)$.
\end{prop}
\begin{proof}
By \ref{unital homs in FA}, the first-factor embedding $\id_B\otimes\eins: B\to B\otimes_{\max} C$ is sequentially split. By the same argument as in the proof of \ref{seq-split-comp}, the composition $(\id_B\otimes\eins)\circ \phi: A\to B\otimes_{\max} C$ has an approximate left-inverse into the double sequence algebra $(A_\infty)_\infty$. Since we have $(\id_B\otimes\eins)\circ\phi = (\phi\otimes\id_C)\circ (\id_A\otimes\eins)$, it follows that also $\id_A\otimes\eins: A\to A\otimes_{\max} C$ has an approximate left-inverse into $(A_\infty)_\infty$. Because $A$ and $C$ are separable, it follows from \ref{weak-seq-split} that $\id_A\otimes\eins$ is sequentially split. The proof is completed with an application of \ref{unital homs in FA}.
\end{proof}

Although not stated directly in these terms, an important result of Toms and Winter on tensorial absorption of strongly self-absorbing \cstar-algebras fits nicely into the picture of sequentially split $*$-homomorphisms: 

\begin{theorem}[cf.~{\cite[2.3]{TomsWinter2007}}] \label{strongly self-absorbing}
Let $A$ be a separable \cstar-algebra and $\CD$ a strongly self-absorbing \cstar-algebra. Then $A$ is $\CD$-stable, that is $A\cong A\otimes\CD$, if and only if the first factor embedding from $A$ into $A\otimes \CD$ is sequentially split.
\end{theorem}

In this way, \ref{unital homs in FA} gives a conceptual reason why Kirchberg's variant \cite[4.11]{Kirchberg2004} is essentially the same result. We note that Toms-Winter's theorem can be viewed as a stronger version of a result found in R{\o}rdam's book \cite[7.2.2]{Rordam2001}.

The following shows that the property of being sequentially split is compatible with inductive limits.

\begin{prop} \label{inductive limits}
Let $\set{A_n, \kappa_n}_{n\in\IN}$ and $\set{B_n, \theta_n}_{n\in\IN}$ be two inductive systems of separable \cstar-algebras with inductive limits $A$ and $B$, respectively. Let $\phi_n: A_n\to B_n$ be a sequence of $*$-homomorphisms compatible with the two inductive systems, i.e.~$\theta_n\circ\phi_n = \phi_{n+1}\circ\kappa_n$ for all $n$. Denote by $\phi: A\to B$ the induced $*$-homomorphism given by the universal property of the inductive limits.
If each of the $*$-homomorphisms $\phi_n$ is sequentially split, then $\phi$ is sequentially split.
\end{prop}
\begin{proof}
For $n\in \IN$, let $\psi_n:B_n\to (A_n)_\infty$ be an approximate left-inverse for $\phi_n$. Let $\eta_n:A_n\to A$ and $\eps_n:B_n\to B$ denote the canonical $*$-homomorphisms. The $\psi_n$ give rise to a $*$-homomorphism
\[
\tilde{\psi}:\prod_{n\in \IN} B_n / \bigoplus_{n\in \IN} B_n \to (A_\infty)_\infty,\ \tilde{\psi}([(b_n)_n])=[((\eta_n)_\infty\circ\psi_n(b_n))_n].
\]
Consider the embedding $\iota:B\into \prod_{n\in \IN} B_n / \bigoplus_{n\in \IN} B_n$ given by the standard construction of the inductive limit, that is, 
\[
\iota( \eps_n(b))=[(\theta_{k-1}\circ\ldots\circ \theta_n(b))_{k\geq n}]
\]
for every $n\in \IN$ and $b\in B_n$. Observe that this notation makes sense, as only the tail of a representing sequence is of interest. Let $\psi:B\to (A_\infty)_\infty$ be the $*$-homomorphism given as the composition of $\tilde{\psi}$ with $\iota$. For $n\in\IN$ and $a\in A_n$, we have that
\[
\begin{array}{lcl}
\psi\circ\phi(\eta_n(a)) & = & \psi\circ\eps_n\circ \phi_n(a)\\
 & = & \tilde{\psi}([(\theta_{k-1}\circ\ldots\circ \theta_n(\phi_n(a)))_{k\geq n}])\\
 & = & [((\eta_k)_\infty\circ\psi_k\circ\theta_{k-1}\circ\ldots\circ \theta_n(\phi_n(a)))_{k\geq n}]
\\
 & = & [((\eta_k)_\infty\circ\kappa_{k-1}\circ\ldots\circ\kappa_n(a))_{k\geq n}] \\
 & = & \eta_n(a) \in (A_\infty)_\infty.
\end{array}
\]
This shows that $\psi\circ\phi$ coincides with the standard embedding of $A$ into $(A_\infty)_\infty$. The claim now follows from \ref{weak-seq-split}. 
\end{proof}

As it turns out, the property of being a sequentially split $*$-homomorphism forces the induced maps on \cstar-algebraic invariants to be very tractable. Moreover, numerous \cstar-algebraic approximation properties pass from the target algebra to the domain algebra. The next three theorems make this explicit and form the main result of this section.

\begin{theorem} \label{sequential dominance properties 1}
Let $A$ and $B$ be two \cstar-algebras. Assume that $\phi: A\to B$ is a sequentially split $*$-homomorphism. Then:
\begin{enumerate}[label=\textup{(\Roman*)}]
\item For each hereditary \cstar-subalgebra $E\subset A$, the restriction $\phi|_E: E\to\quer{\phi(E)B\phi(E)}$ is sequentially split. \label{sdp1:1}
\item For each ideal $J$ of $A$, the restriction $\phi|_J: J\to\quer{B\varphi(J)B}$ and the induced map $\phi_{\mathrm{mod}\, J}: A/J\to B/\quer{B\varphi(J)B}$ are sequentially split. \label{sdp1:2}
\item The induced map between the ideal lattices, given by $J\mapsto \quer{B\varphi(J)B}$, is injective. \label{sdp1:3}
\item Assume that $\rho:C\to D$ is another sequentially split $*$-homomorphism. Then $\phi\otimes\rho:A\otimes_{\max} C\to B\otimes_{\max} D$ is sequentially split. \label{sdp1:4}
\item The induced map between the Cuntz semigroups $\Cu(A)\to\Cu(B)$ given by $\langle a \rangle_A \mapsto \langle \phi(a) \rangle_B$ is an order-embedding. \label{sdp1:5}
\item The induced map on $K$-theory $K_*(\phi): K_*(A)\to K_*(B)$ is injective. The same is true for $K$-theory with coefficients in $\IZ_n$ for all $n\geq 2$. \label{sdp1:6}
\item The induced map between the simplices of tracial states $T(\phi): T(B)\to T(A)$ given by $\tau\mapsto \tau\circ\phi$ is surjective. \label{sdp1:7}
\end{enumerate}
\end{theorem}
\begin{proof}
For what follows, $\psi:B\to A_\infty$ is an approximate left-inverse for $\phi$.

\ref{sdp1:1}: Let $E\subset A$ be a hereditary subalgebra. Then $E=\quer{EAE}$, and in particular, $e_1ae_2\in E$ for any $e_1,e_2\in E$ and $a\in A$.
We conclude that for $e_1,e_2\in E$ and $b\in B$,
\[
\psi(\phi(e_1)b\phi(e_2))=e_1\psi(b)e_2\in E_\infty.
\]
So indeed, the restriction $\phi|_E: E\to\quer{\phi(E)B\phi(E)}$ is sequentially split.

\ref{sdp1:2}: Let $J\subset A$ be an ideal and consider the induced $*$-homomorphism $\phi:J\to \overline{B\phi(J)B}$. As $J\subset A$ is an ideal, we conclude that for $j\in J$ and $b,b'\in B$,
\[
\psi(b\phi(j)b')=\psi(b)j\psi(b')\in J_\infty.
\]
Hence, $\phi|_J : J\to \overline{B\phi(J)B}$ is sequentially split. This also shows that $\psi$ induces a $*$-homomorphism
\[
\tilde{\psi}:B/\overline{B\phi(J)B}\to A_\infty/J_\infty\cong (A/J)_\infty,\ [(b_n)_n]\mapsto [(\psi(b_n))_n].
\]
It is clear from the construction that $\tilde{\psi}\circ \phi_{\mathrm{mod}\, J}$ recovers the standard embedding of $A/J$ into its sequence algebra. This proves \ref{sdp1:2}.

\ref{sdp1:3}: Let $I$ and $J$ be two different ideals in $A$. We may assume that $I\nsubseteq J$. By \ref{sdp1:2}, $\psi(\quer{B\varphi(J)B}) \subset J_\infty$ and $I \subset \psi(\quer{B\varphi(I)B}) \subset I_\infty$. As $I \nsubseteq J_\infty$, the ideals $\quer{B\varphi(I)B}$ and $\quer{B\varphi(J)B}$ have to be different. This concludes the proof of \ref{sdp1:3}.

\ref{sdp1:4}: Consider the canonical $*$-homomorphism $A_\infty\otimes_{\max} C_\infty\to (A\otimes_{\max} C)_\infty$ and note that it is compatible with the canonical inclusion of $A\otimes_{\max} C$ into $(A \otimes_{\max} C)_\infty$. There exists a commutative diagram
\[
\xymatrix{
A\otimes_{\max} C \ar[rr] \ar[dr]_{\phi\otimes \rho} & & A_\infty\otimes_{\max} C_\infty \ar[r] & (A\otimes_{\max} C)_\infty \\
 & B\otimes_{\max} D \ar[ru] \ar@{-->}[rru]
}
\]
which shows that $\phi\otimes \rho$ is sequentially split.

\ref{sdp1:5}: Let $a,b\in A\otimes \CK$ be positive elements satisfying $\Cu(\phi)(\langle a \rangle_A) \leq \Cu(\phi)(\langle b \rangle_A) \in \Cu(B)$. Applying $\Cu(\psi)$ to both sides of the inequality, we obtain that $\langle a \rangle_{A_\infty} \leq \langle b \rangle_{A_\infty} \in \Cu(A_\infty)$. Given $\eps > 0$, we find some $v\in A_\infty \otimes \CK$ with the property that $vbv^*=_\eps a$ in $A_\infty \otimes \CK$. We may assume that $v \in M_n(A_\infty) \cong (M_n(A))_\infty$ for some $n\in \IN$. We therefore find some $w \in M_n(A)$ with $wbw^*=_{2\eps} a$ in $A\otimes\CK$.  Hence, $\langle a \rangle_A \leq \langle b \rangle_A \in \Cu(A)$, showing that $\Cu(\varphi)$ is an order embedding.

\ref{sdp1:6}: As $\id_{\CC_0(\IR)} \otimes \phi: \CC_0(\IR)\otimes A \to \CC_0(\IR)\otimes B$ is sequentially split by \ref{sdp1:4}, it suffices to prove the claim for $K_0$. Moreover, we may assume that $A,B$ and $\phi$ are unital. Indeed, as $(A_\infty)^\sim \subset (A^\sim)_\infty$ canonically, $\phi^\sim:A^\sim \to B^\sim$ is sequentially split. The assertion in the unital case yields that $K_0(\phi^\sim)$ is injective. By commutativity of the diagram
\[
\xymatrix{
K_0(A) \ar@{^(->}[r] \ar[d]_{K_0(\phi)} & K_0(A^\sim) \ar@{^(->}[d]^{K_0(\phi^\sim)} \\
K_0(B) \ar@{^(->}[r] & K_0(B^\sim)
}
\]
we get that $K_0(\phi)$ is injective. Let $A, B$ and $\phi$ be unital and let $p,q\in M_n(A)$ be projections with $K_0(\phi)([p] - [q]) = 0 \in K_0(B)$. Then 
\[
K_0(\psi) \circ K_0(\phi)([p] - [q]) = [p] - [q] = 0 \in K_0(A_\infty).
\]
By the definition of the $K_0$-group, we find $k,l\in \IN$ such that
\[
p \oplus \eins_k \oplus 0_l \sim_{\mathrm{MvN}} q \oplus \eins_k \oplus 0_l\ \text{in}\ M_n(A_\infty) \cong (M_n(A))_\infty. 
\]
Since the relation of being a partial isometry is weakly stable, this implies that
\[
p \oplus \eins_k \oplus 0_l \sim_{\mathrm{MvN}} q \oplus \eins_k \oplus 0_l \ \text{in}\ M_n(A).
\]
This shows that $[p] - [q] = 0 \in K_0(A)$, and we conclude that $K_0(\phi)$ is injective.

Injectivity of the induced map in $K$-theory with coefficients in $\IZ_n$ follows from the fact that $\varphi\otimes \id_{\CO_{n+1}}:A\otimes \CO_{n+1}\to B\otimes \CO_{n+1}$ is sequentially split by \ref{sdp1:4}, see \cite[6.4]{Schochet1984}.

\ref{sdp1:7}: Let $\tau\in T(A)$ be a tracial state on $A$ and $\omega$ a free ultrafilter on $\IN$. Consider the induced tracial state
\[
\tau_\omega : A_\infty \to \IC,\ \tau_\omega([(a_n)_n])=\lim_{n\to \omega} \tau(a_n).
\]
Then $\tau'=\tau_\omega \circ \psi$ is a tracial state on $B$ satisfying $\tau' \circ \phi = \tau$. We conclude that $T(\phi)$ is surjective. 
\end{proof}

\begin{theorem} \label{sequential dominance properties 2}
Let $A$ and $B$ be two separable \cstar-algebras. Assume that $\phi: A\to B$ is a sequentially split $*$-homomorphism. The following properties pass from $B$ to $A$:
\begin{enumerate}[label=\textup{(\arabic*)}]
\item simplicity. \label{sdp2:1}
\item nuclearity. In fact, $A$ is nuclear if $\phi$ is nuclear. \label{sdp2:2}
\item nuclear dimension at most $r\in\IN$. In fact, $\dimnuc(A)\leq\dimnuc(\phi)$. \label{sdp2:3}
\item decomposition rank at most $r\in\IN$. In fact, $\dr(A)\leq\dr(\phi)$. \label{sdp2:4}
\item absorbing a given strongly self-absorbing \cstar-algebra $\CD$. \label{sdp2:5}
\item If $A, B$ and $\phi$ are unital: being isomorphic to a given strongly self-absorbing \cstar-algebra. \label{sdp2:6}
\item If $A, B$ and $\phi$ are unital: approximate divisibility. \label{sdp2:7}
\item being purely infinite. \label{sdp2:8}
\item having almost unperforated Cuntz semigroup. \label{sdp2:9}
\item If $A, B$ and $\phi$ are unital: having strict comparison of positive elements. \label{sdp2:10}
\item real rank zero. \label{sdp2:11}
\item stable rank one or almost stable rank one in the sense of \cite[3.1]{Robert2013}. \label{sdp2:12}
\item being locally approximated by or being expressible as an inductive limit of \cstar-algebras in a class $\FC$ consisting of weakly semiprojective \cstar-algebras whose quotients can all be locally approximated by \cstar-algebras in $\FC$. \label{sdp2:13}
\item being either UHF, AF, AI, A$\IT$ or being expressible as an inductive limit of 1-NCCW complexes.
\label{sdp2:14}
\item If $A, B$ and $\phi$ are unital: being simple, exact and having tracial rank zero, tracial rank at most one or having generalized tracial rank at most one in the sense of \cite[Section 9]{GongLinNiu2015}. \label{sdp2:15}
\item stability under tensoring with the compacts $\CK$. \label{sdp2:16}
\end{enumerate}
We note that separability is only necessary in order to prove \ref{sdp2:5}, \ref{sdp2:13}, \ref{sdp2:14} and \ref{sdp2:15}. For \ref{sdp2:16}, $\sigma$-unitality is sufficient.
\end{theorem}

\begin{proof}
For what follows, $\psi:B\to A_\infty$ is an approximate left-inverse for $\phi$.

\ref{sdp2:1}: This is an immediate consequence of \ref{sequential dominance properties 1}\ref{sdp1:3}.

\ref{sdp2:2}: It is a well-known consequence of the Choi-Effros lifting theorem \cite{ChoiEffros1976} that a \cstar-algebra is nuclear if and only if its standard embedding from $A$ into $A_\infty$ is nuclear.
Now if $\phi: A\to B$ is both nuclear and sequentially split, then this implies that the standard embedding from $A$ into $A_\infty$ is also necessarily nuclear.
  
\ref{sdp2:3} and \ref{sdp2:4}: It is well-known that the nuclear dimension of a \cstar-algebra is identical to the nuclear dimension of its standard embedding into its sequence algebra, see \cite[2.5]{TikuisisWinter2014}. The same is true for decomposition rank. Thus the same argument as in \ref{sdp2:2} holds as the standard embedding of $A$ factorizes through $\phi$, and thus the claim follows.

\ref{sdp2:5}: Assume that $B\cong B\otimes \CD$. As pointed out in \ref{strongly self-absorbing}, this is equivalent to the first factor embedding $B\into B\otimes \CD$ being  sequentially split. We obtain a commutative diagram
\[
	\xymatrix{
	A \ar[rrrr] \ar[rd]_{\id_A\otimes \eins} \ar[rrd]^\phi & & & & (A_\infty)_\infty \\
	 & A\otimes \CD  \ar[rd]_{\phi\otimes\id_\CD} \ar@{-->}[rrru]& B \ar[r] \ar[d]_/-0.3cm/{\id_B\otimes \eins} & B_\infty \ar[ur] & \\
	 & & B\otimes\CD \ar[ru] & &
	}
\]
An application of \ref{weak-seq-split} now yields that $\id_A\otimes \eins:A\into A\otimes \CD$ is sequentially split. Hence, $A\cong A\otimes \CD$ by \ref{strongly self-absorbing}.

\ref{sdp2:6}: Suppose that $A, B$ and $\phi$ are unital. Assume that $B$ is strongly self-absorbing. In particular, $B$ has approximately inner half-flip by \cite[1.5]{TomsWinter2007}. Let $u_n\in \CU(B\otimes B)$ be a sequence of unitaries with $u_n(b\otimes\eins)u_n^*\to\eins\otimes b$ for all $b\in B$.
As $\phi\otimes \phi: A\otimes A\to B\otimes B$ is sequentially split by \ref{sequential dominance properties 1}\ref{sdp1:4}, choose an approximate left-inverse $\eta: B\otimes B\to (A\otimes A)_\infty$ for $\phi \otimes \phi$. The sequence of unitaries $v_n=\eta(u_n)\in (A\otimes A)_\infty$ then satisfies $v_n(a\otimes\eins)v_n^*\to\eins\otimes a$ for all $a\in A$. After lifting each $v_n$ to a sequence of unitaries, a simple diagonal sequence argument now shows that $A$ also has approximately inner half-flip.
 
As $B$ is strongly self-absorbing and admits a unital embedding into its own central sequence algebra, we can also embed $A$ unitally into the central sequence algebra of $B$. So \cite[7.2.2]{Rordam2001} implies that $B$ absorbs $A$ tensorially. On the other hand, we can use \ref{sdp2:5} to see that $A$ absorbs $B$ tensorially. We conclude that $A$ and $B$ are isomorphic.

\ref{sdp2:7}: Suppose that $A, B$ and $\phi$ are unital. If $B$ is approximately divisible, then by definition, there exists a unital $*$-homomorphism $M_2\oplus M_3\to B_\infty\cap B'$. By \ref{unital homs in FA seq split}, we obtain a unital $*$-homomorphism $M_2\oplus M_3\to A_\infty\cap A'$, which implies that $A$ is approximately divisible.

\ref{sdp2:8}: Assume that $B$ is purely infinite (see \cite[1.3]{Kirchberg2000} for a definition). Let $a\in A$ be a positive element and let $a' \in A$ be another positive element contained in the ideal generated by $a$. Since $B$ is purely infinite, there exist $b_n \in B$, $n \in \IN$, such that $\lim_{n\to \infty} b_n \phi(a) b_n^* = \phi(a')$. Applying $\psi:B\to A_\infty$ to both sides of the equality, we get that $\lim_{n\to\infty} \psi(b_n) a \psi(b_n)^* = a'$ in $A_\infty$. A simple diagonal sequence argument yields elements $c_n \in A$, $n \in \IN$, such that $\lim_{n\to\infty} c_n ac_n^*=a'$.

Assume that $A$ admits a character $\chi:A\to \IC$. Let $\omega$ be a free ultrafilter on $\IN$. Then
\[
\chi_\omega: A_\infty \to \IC,\ \chi_\omega([(a_n)_n])=\lim_{n\to\omega} \chi(a_n),
\]
defines a character on $A_\infty$ satisfying $\chi_\omega \circ \psi \circ \phi = \chi$. This implies that $\chi_\omega\circ \psi: B\to \IC$ is a character, which contradicts the fact that $B$ is purely infinite. We have proven \ref{sdp2:8}.

\ref{sdp2:9}: Let $x,y\in \Cu(A)$, $n\in \IN$ with $(n+1)x\leq ny$. Then $(n+1)\Cu(\phi)(x) \leq n\Cu(\phi)(y)$. As $\Cu(B)$ is almost unperforated (see \cite[3.1]{Rordam2004} for a definition), $\Cu(\phi)(x) \leq \Cu(\phi)(y)$. As $\Cu(\phi)$ is an order embedding by \ref{sequential dominance properties 1}\ref{sdp1:5}, we conclude that $x \leq y$. This shows that $\Cu(A)$ is almost unperforated. 

\ref{sdp2:10}: It follows from \cite[3.2]{Rordam2004} and \cite[5.7]{AraPereraToms2011} that a \cstar-algebra $C$ has strict comparison (see \cite[7.6.4]{AntoinePereraThiel2014} for a definition) if and only if the uncompleted Cuntz semigroup $W(C)$ is almost unperforated. Moreover, for any \cstar-algebra, the uncompleted Cuntz semigroup is almost unperforated if and only if the Cuntz semigroup is unperforated, see \cite[7.6.4]{AntoinePereraThiel2014}. The claim now follows from \ref{sdp2:9}.

\ref{sdp2:11}: Let $a\in A$ be a self-adjoint element. If $B$ has real rank zero, then there are mutually orthogonal projections $p_1,\ldots,p_n$ and $\lambda_1,\ldots,\lambda_n\in \IR$ such that
\[
\phi(a)=_\eps \sum\limits_{j=1}^n \lambda_j p_j.
\]
Applying $\psi:B\to A_\infty$, we obtain
\[
a=_\eps \sum\limits_{j=1}^n \lambda_j \psi(p_j).
\]
Given $m\in \IN$, a diagonal sequence argument yields positive elements $x_1,\ldots,x_n\in A$ satisfying
\[
a=_{2\eps} \sum\limits_{j=1}^n \lambda_j x_j, \quad x_k^2=_{\frac{1}{m}} x_k\quad \text{and}\quad x_kx_l =_{\frac{1}{m}} 0
\]
for $1\leq k,l\leq n$ with $k\neq l$. Since $\IC^n$ is weakly semiprojective, we find mutually orthogonal projection $q_1,\ldots, q_n\in A$ satisfying
\[
a=_{3\eps} \sum\limits_{j=1}^n \lambda_j q_j.
\]
Hence, $A$ has real rank zero.

\ref{sdp2:12}: Assume that $B$ has stable rank one. Since $(A_\infty)^\sim\subset (A^\sim)_\infty$ canonically, the map $\phi^\sim:A^\sim\to B^\sim$ is sequentially split. Given $a\in A^\sim$, we find some invertible element $b\in B^\sim$ such that $\phi^\sim(a)=_\eps b$. Then $\psi^\sim(b)\in (A^\sim)_\infty$ is an invertible element satisfying $a=_\eps \psi^\sim(b)$. As any invertible element in any sequence algebra lifts to a sequence of invertible elements, we can represent $\psi^\sim(b)$ by a bounded sequence of invertible elements in $A^\sim$. Picking a suitable member of this sequence yields some invertible $c\in A^\sim$ with $a=_{2\eps}c$. Since $\eps>0$ was arbitrary, this shows that $A$ has stable rank one.

Now assume that $B$ has almost stable rank one in the sense of Robert, see \cite[3.1]{Robert2013}. Then the identical argument as above yields that any $a\in A$ can be approximated by invertibles in $A^\sim$. Moreover, by \ref{sequential dominance properties 1}\ref{sdp1:1}, the same holds if we replace $A$ by some hereditary subalgebra. This shows that $A$ has almost stable rank one.

\ref{sdp2:13}: Assume that $B$ is locally approximated by \cstar-algebras in $\FC$. Let $F\fin A$ be a finite subset and $\eps>0$. Find some \cstar-subalgebra $C \subset B$ such that $C \in \FC$ and $\varphi(F)\subset_\eps C$. Then $F\subset_\eps \psi(C)$ in $A_\infty$. As $C$ is weakly semiprojective, $\psi|_{C}:C \to A_\infty$ lifts to a $*$-homomorphism $C\to \ell^\infty(\IN,A)$. Composing it with the canonical projection onto a suitable coordinate $\ell^\infty(\IN,A)\to A$, we obtain a $*$-homomorphism $\kappa:C\to A$ satisfying $F\subset_{2\eps} \kappa(C)$. As $\kappa(C)$ can be locally approximated by \cstar-algebras in $\FC$, we find a \cstar-subalgebra $D\subset \kappa(C)\subset A$ such that $D \in \FC$ and $F\subset_{3\eps}D$. Hence, $A$ can be locally approximated by \cstar-algebras in $\FC$. The corresponding statement about expressing $A$ as an inductive limit now follows directly from \ref{perturb weakly semiproj}.

\ref{sdp2:14}: 
Recall that 1-NCCW complexes are semiprojective, see \cite[6.2.2]{EilersLoringPedersen1998} and also \cite[3.4]{Enders2014}. The same is true for finite dimensional \cstar-algebras. Applying \ref{sdp2:13} to the class $\FC$ of all matrix algebras, we therefore conclude that the property of being UHF passes from $B$ to $A$. Similarly, if $\FC$ is the class of all finite dimensional \cstar-algebras, we get that $A$ is an AF-algebra if this is true for $B$.

Assume now that $B$ is an A$\IT$-algebra. Let $\FC$ be the class of all \cstar-algebras of the form $\CF_1 \oplus \CF_2 \otimes \CC(\IT)$, where $\CF_1$ and $\CF_2$ are finite-dimensional. Every quotient of a circle algebra can be locally approximated by \cstar-algebras in $\FC$. Hence, $A$ can be written as an inductive limit of \cstar-algebras in $\FC$ by \ref{sdp2:13}. As every \cstar-algebra in $\FC$ is also a quotient of a circle algebra, it follows from \cite[3.2.3]{Rordam2001} that $A$ is A$\IT$.

If $B$ is an AI-algebra then it is also an A$\IT$-algebra. Indeed, AI-algebras are exactly the A$\IT$-algebras with trivial $K_1$-group, see \cite[3.2.17]{Rordam2001}. By \ref{sequential dominance properties 1}\ref{sdp1:6}, we get that $K_1(A) = 0$ and we conclude that $A$ is an AI-algebra. 

Lastly, assume that $B$ is expressible as an inductive limit of 1-NCCW complexes. Basically the same proof as in \cite[3.20]{GongLinNiu2015} shows that every quotient of a 1-NCCW complex can be locally approximated by \cstar-algebras of the form $C_1 \oplus C_2$, where $C_1$ is finite dimensional and $C_2$ is a 1-NCCW complex. Every \cstar-algebra of this form is the image of a split surjection starting from some 1-NCCW complex. We therefore conclude that $A$ is expressible as an inductive limit of 1-NCCW complexes. This shows \ref{sdp2:14}.

\ref{sdp2:15}: Let $\CS$ be either one of the following classes of \cstar-algebras:
\begin{itemize}
\item all finite-dimensional \cstar-algebras;
\item all \cstar-algebras isomorphic to $\CF_1\oplus\CF_2\otimes\CC[0,1]$, where $\CF_1$ and $\CF_2$ are finite-dimensional;
\item All unital 1-NCCW-complexes (also known as Elliott-Thomson building blocks \cite[Section 3]{GongLinNiu2015}) and all finite-dimensional \cstar-algebras.
\end{itemize}
Assume now that $B$ is simple, exact and $B\in \mathrm{TA}\CS$, see \cite[9.4]{GongLinNiu2015}. Then $A$ is simple by \ref{sdp2:1}, and moreover exact because $\phi$ is injective. This implies that $T(A)=QT(A)$ by Haagerup's theorem \cite{Haagerup2014}. Let us show that also $A\in\mathrm{TA}\CS$. We note (as in the proof of \ref{sdp2:14}) that the class $\CS$ consists of weakly semiprojective \cstar-algebras, and that quotients of \cstar-algebras in $\CS$ can be locally approximated by \cstar-algebras in $\CS$. By \cite[9.11]{GongLinNiu2015}, $B$ has strict comparison of positive elements, and so has $A$ by \ref{sdp2:10}. In particular, it suffices to show (cf.~\cite[6.15]{Lin2001}) that for every $\eps>0$ and $F\fin A$, there exists $C\subset A$ with $C\in\CS$ such that
\begin{itemize}
\item $\|[\eins_C,x]\|\leq\eps$;
\item $\dist(\eins_C x \eins_C, C)\leq\eps$;
\item $\tau(\eins_C)\geq 1-\eps$
\end{itemize}
for all $x\in F$ and $\tau\in T(A)$.
Since $B\in\mathrm{TA}\CS$, it follows that we can find $C_1\subset B$ with $C_1\in\CS$ such that
\begin{itemize}
\item $\|[\eins_{C_1},\phi(x)]\|\leq\eps$;
\item $\dist(\eins_{C_1} \phi(x) \eins_{C_1}, C_1)\leq\eps$;
\item $\tau(\eins_{C_1})\geq 1-\eps$
\end{itemize}
for all $x\in F$ and $\tau\in T(B)$. Consider the restriction $\psi|_{C_1}: C_1\to A_\infty$, and use weak semiprojectivity to lift this to a sequence of $*$-homomorphisms $\kappa_n: C_1\to A$. As $\psi\circ\phi$ coincides with the standard embedding of $A$ into $A_\infty$ we obtain
\begin{itemize}
\item $\dst \limsup_{n\to\infty} \|[\kappa_n(\eins_{C_1}),x]\|\leq\eps$;
\item $\dst \limsup_{n\to\infty} \dist(\kappa_n(\eins_{C_1}) x \kappa_n(\eins_{C_1}), \kappa_n(C_1) )\leq\eps$;
\item $\dst \liminf_{n\to\infty} \min_{\tau\in T(A)}\tau(\kappa_n(\eins_{C_1}))\geq 1-\eps$
\end{itemize}
for all $x\in F$ \footnote{Note that more generally, if $a\in A_\infty$ is a positive contraction satisfying $\tau(a)\geq c\geq 0$ for all $\tau\in T(A_\infty)$, then any representing sequence $(a_n)_n$ of positive contractions satisfies $\dst\liminf_{n\to\infty}~\min_{\tau\in T(A)} \tau(a_n) \geq c$.}.
In particular, we can pick a member of this sequence $\kappa: C_1\to A$ of $*$-homomorphisms satisfying
\begin{itemize}
\item $\|[\kappa(\eins_{C_1}),x]\|\leq 2\eps$;
\item $\dst \dist(\kappa(\eins_{C_1}) x \kappa(\eins_{C_1}), \kappa(C_1) )\leq2\eps$;
\item $\tau(\kappa(\eins_{C_1}))\geq 1-2\eps$
\end{itemize}
for all $x\in F$ and $\tau\in T(A)$. Using the previously mentioned fact that quotients of \cstar-algebras in $\CS$ are locally approximated by \cstar-algebras in $\CS$, we can find $C\in\CS$ with $C\subset\kappa(C_1)$ unitally, satisfying
\begin{itemize}
\item $\|[\eins_{C},x]\|\leq 2\eps$;
\item $\dst \dist( \eins_{C} x \eins_{C}, C )\leq 3\eps$;
\item $\tau(\eins_{C})\geq 1-2\eps$
\end{itemize}
for all $x\in F$ and $\tau\in T(A)$. Since $\eps>0$ and $F\fin A$ were arbitrary, this shows that indeed $A\in\mathrm{TA}\CS$.

\ref{sdp2:16}: By \cite[2.1 and 2.2]{HjelmborgRordam1998}, a $\sigma$-unital \cstar-algebra $C$ is stable if and only if for each positive $c\in C$ and $\eps>0$, there exists some $d\in C$ such that $d^*d=_\eps c$ and $\| d^*ddd^*\| \leq \eps$. Assume that $B$ is stable and let $a\in A$ be positive. Then there exists some $x\in B$ such that $x^*x=_\eps \phi(a)$ and $\| x^*xxx^* \| \leq \eps$. If $y = \psi(x) \in A_\infty$, we therefore get that $y^*y=_\eps a$ and $\| y^*yyy^*\| \leq \eps$. Picking a suitable member $z \in A$ of a representing sequence for $y$, we can arrange that $z^*z=_{2\eps} a$ and $\|z^*zzz^*\|\leq 2\eps$. We conclude that $A$ is stable.
\end{proof}

Somewhat less obvious than most of the properties listed in \ref{sequential dominance properties 2}, it turns out that the UCT together with nuclearity is inherited under sequential dominance as well. We note that the key arguments in the proof below are a combination of a nearly identical argument due to  Kirchberg in \cite{Kirchberg}, where he reduces the UCT problem to the purely infinite setting, and a nearly identical argument due to Dadarlat in \cite{Dadarlat2009}, where he gives a simplified proof of a special case of his theorem \cite{Dadarlat2003} that the UCT is a local property.

\begin{theorem} \label{UCT}
Let $A$ and $B$ be separable \cstar-algebras. Assume that $A$ is sequentially dominated by $B$. If $B$ is nuclear and satisfies the UCT, then so does $A$.
\end{theorem}
\begin{proof}
Assume that $B$ is nuclear and satisfies the UCT. Let $\phi:A\to B$ be a sequentially split $*$-homomorphism. By \ref{sequential dominance properties 2}\ref{sdp2:2}, we already know that $A$ is nuclear. As $\phi^\sim:A^\sim\to B^\sim$ is also sequentially split, we may assume that $A,B$ and $\phi$ were unital to begin with. By passing to the sequentially split (see \ref{sequential dominance properties 1}\ref{sdp1:4}) $*$-homomorphism $\phi\otimes \id_{\Ost}:A\otimes \Ost\to B\otimes \Ost$, we may also assume that $A$ and $B$ are $\Ost$-absorbing, since $\Ost$ is $KK$-equivalent to $\IC$. As in \cite[Theorem I]{Kirchberg} or \cite[4.17]{BarlakSzabo2014}, we construct unital Kirchberg algebras $A^\sharp$ and $B^\sharp$, which are $KK$-equivalent to $A$ and $B$, respectively. Since we need this construction to be compatible with the unital injective $*$-homomorphism $\phi:A\to B$, we recall it here, see also \cite[proof of 5.2]{Dadarlat2009}. By \cite{KirchbergPhillips2000}, there exists a unital embedding $\kappa:B\into \CO_2$. As it is well-known that $\CO_2$ embeds into $\Ost$ unitally, we find a unital embedding $\iota:\CO_2\into A$, and write $s_1,s_2\in A$ for the images of the canonical generators of $\CO_2$. Define a $*$-endomorphism
\[
\psi_A : A\to A,\ \psi_A(x)=s_1 x s_1^* + s_2 (\iota\circ\kappa\circ \phi)(x) s_2^*.
\]
Clearly, $\phi\circ \iota :\CO_2\into B$ is a unital embedding, and we define a $*$-endomorphism
\[
\psi_B : B\to B,\ \psi_B(x)=\phi(s_1) x \phi(s_1)^* + \phi(s_2) (\phi\circ\iota\circ\kappa)(x) \phi(s_2)^*.
\]
Let $A^\sharp=\dst\lim_{\longrightarrow}\set{A,\psi_A}$ and $B^\sharp=\dst\lim_{\longrightarrow}\set{B,\psi_B}$ denote the corresponding stationary inductive limits. Clearly, $A^\sharp$ and $B^\sharp$ are again separable, unital, nuclear and $\CO_\infty$-absorbing \cstar-algebras. Since for all $x\neq 0$, the element
\[
s_2^*\psi_A(x)s_2 = \iota\circ\kappa\circ\phi(x)
\]
is the image of the full element $\kappa\circ\phi(x)\in\CO_2$, it follows that $\psi_A(x)$ is also full. Hence $A^\sharp$ is simple. The same argument for $B$ shows that $B^\sharp$ is simple.

It is immediate that $KK(\psi_A)=\eins+KK(\iota\circ\kappa\circ\phi)=\eins$, since $\iota\circ\kappa\circ\phi$ factors through $\CO_2$. In particular, the connecting maps of this inductive system induce $KK$-equivalences. Hence it follows that the canonical embedding $\psi_{A,\infty}: A\to A^\sharp$ induces a $KK$-equivalence. This is implied by \cite[2.4]{Dadarlat2009}, which basically boils down to plugging in the Milnor sequence \cite[21.5.2]{Blackadar1998} for the functor $KK(~\_~,B)$ in this situation. By a similar argument, we also get that the canonical embedding $\psi_{B,\infty}: B\to B^\sharp$ is a $KK$-equivalence. We conclude that $A^\sharp$ and $B^\sharp$ are unital Kirchberg algebras $KK$-equivalent to $A$ and $B$, respectively.

Hence it suffices to show that $A^\sharp$ satisfies the UCT.
Note that by construction, the following diagram commutes and thus induces a $*$-homomorphism $\phi^\sharp$ via
\[
	\xymatrix{
	A \ar[r]^{\psi_A} \ar[d]^\phi & A \ar[r] \ar[d]^\phi & \cdots \ar[r] & A^\sharp \ar@{-->}[d]^{\phi^\sharp}\\
	B \ar[r]^{\psi_B} & B \ar[r] & \cdots \ar[r] & B^\sharp
	}
\]
By \ref{inductive limits}, it follows that $\phi^\sharp: A^\sharp\to B^\sharp$ is sequentially split.
Since $B^\sharp$ is a UCT Kirchberg algebra, it is expressible as an inductive limit of UCT Kirchberg algebras with finitely generated  $K$-theory, see \cite[8.4.13]{Rordam2001}. These \cstar-algebras are known to be weakly semiprojective, see \cite{Lin2007, Spielberg2007}. Using \ref{sequential dominance properties 2}\ref{sdp2:13}, we therefore conclude that $A^\sharp$ is expressible as an inductive limit of UCT Kirchberg algebras with finitely generated $K$-theory. This implies that $A^\sharp$ indeed satisfies the UCT, which finishes the proof.
\end{proof}


\section{Sequentially split homomorphisms: The equivariant case}

\begin{nota}
Let $G$ be a locally compact group, $A$ a \cstar-algebra and $\alpha:G\curvearrowright A$ a point-norm continuous action.  Componentwise application of $\set{\alpha_g}_{g\in G}$ on representative sequences yields a (discrete) $G$-action $\alpha_\infty$ on $A_\infty$. If $B\subset A_\infty$ is a (globally) $\alpha_\infty$-invariant \cstar-subalgebra, then we get induced actions $\tilde{\alpha}_\infty$ on $F(B,A_\infty)$ and also $\CM(D_{B,A_\infty})$.\footnote{This notation makes sense because if $B$ is $\sigma$-unital, then \ref{surjection normalizer multiplier}\ref{snm:2} implies that the induced action on $F(B,A_\infty)$ may be viewed as a restriction of the action induced on $\CM(D_{B,A_\infty})$.}
These actions are in general not continuous. However, we may restrict to the continuous parts of these actions, for instance for $\alpha_\infty$ on $A_\infty$ we consider
\[
A_{\infty,\alpha} = \set{ x\in A_\infty ~|~ [g\mapsto\alpha_{\infty,g}(x)]~\text{is continuous} }.
\]
In this way, we obtain \cstar-dynamical systems $(A_{\infty,\alpha}, \alpha_\infty)$, $(\CM_{\alpha}(D_{B,A_\infty}),\tilde{\alpha}_\infty)$ and $(F_\alpha(B,A_\infty), \tilde{\alpha}_\infty)$. 
For brevity, we denote $F_\alpha(A,A_\infty) = F_{\infty,\alpha}(A)$.
\end{nota}

\begin{rem} \label{multiplication FA continuous}
The natural $*$-homomorphism in \ref{multiplication FA} restricts to a $*$-homomor\-phism
\[
F_{\infty,\alpha}(A)\otimes_{\operatorname{max}} A \to A_{\infty,\alpha},\ (b+\ann(A,A_\infty))\otimes a \mapsto ba,
\]
which clearly is $\tilde{\alpha}_\infty\otimes \alpha$-to-$\alpha_\infty$-equivariant. Observe that this $*$-homomorphism indeed maps into $A_{\infty,\alpha}$, since the action on the left-hand side is point-norm continuous.
\end{rem}

\begin{defi} Let $A$ and $B$ be separable \cstar-algebras and $G$ a locally compact group. Let $\alpha:G\curvearrowright A$ and $\beta:G\curvearrowright B$ be two continuous actions. An equivariant $*$-homomorphism $\phi:(A,\alpha)\to (B,\beta)$ is called ($G$-)(equivariantly) sequentially split, if there exists a commutative diagram of equivariant $*$-homomorphisms of the form
\[
\xymatrix{
(A,\alpha) \ar[dr]_\phi \ar[rr] && (A_{\infty,\alpha},\alpha_\infty) \\
& (B,\beta) \ar[ur] &
}
\]
where the horizontal map is the canonical inclusion. If $\psi:(B,\beta) \to (A_{\infty,\alpha},\alpha_\infty)$ is an equivariant $*$-homomorphism fitting into the above diagram, then we say that $\psi$ is an equivariant approximate left-inverse for $\phi$.
\end{defi}

Similarly as for path algebras in \cite[1.8]{GuentnerHigsonTrout2000}, it turns out that continuous elements in the sequence algebra have a particularly strong continuity property also on the level of their representatives. We note that the technical proof appearing in the initial preprint version of this paper turned out to be redundant upon discovering a much more general result due to Brown \cite{Brown00}.

\begin{lemma}[cf.\ {\cite[Theorem 2]{Brown00}}] \label{continuous lifts}
Let $G$ be a locally compact group, $A$ a \cstar-algebra and $\alpha: G \curvearrowright A$ a continuous action. Let $x \in A_{\infty}$ and $(x_n)_n \in \ell^\infty(\IN,A)$ a representing sequence. Then $x\in A_{\infty,\alpha}$ if and only if $(x_n)_n$ is a continuous element with regard to the action induced on $\ell^\infty(\IN,A)$, i.e.\ the map $g\mapsto\big( \alpha_g(x_n) \big)_n\in\ell^\infty(\IN,A)$ is continuous. In particular, if $x\in A_{\infty,\alpha}$, then the following holds: For any $g_0 \in  G$ and $\delta > 0$, there exists a neighbourhood $U$ of $g_0$ such that
\[
\sup_{k \in \IN}~ \sup_{g\in U} \| \alpha_g(x_k) - \alpha_{g_0}(x_k) \| \leq \delta.
\]
\end{lemma}

\begin{lemma} \label{uniform norm}
Let $G$ be a locally compact group, $A$ a \cstar-algebra and $\alpha:G \curvearrowright A$ a continuous action. Let $x \in A_{\infty,\alpha}$ and $(x_n)_n \in \ell^\infty(\IN,A)$ a representing sequence. Then the following statement holds: For every compact set $K \subset G$ and $\delta > 0$, there exists some $n_0\in\IN$ such that for all $g \in K$,
\[
\sup_{k \geq n_0} \| \alpha_g(x_k) - x_k \| \leq \| \alpha_{\infty,g}(x) - x \| + \delta.
\]
\end{lemma}
\begin{proof}
Let $K \subset G$ be compact and $\delta > 0$. Given $g_0 \in K$, we apply \ref{continuous lifts} and find some open neighbourhood $U_0$ of $g_0$ such that 
\[
\sup_{k \in \IN}~ \sup_{g\in U_0} \| \alpha_g(x_k) - \alpha_{g_0}(x_k) \| \leq \delta/3.
\]
By passing to a possibly smaller open neighbourhood of $g_0$ and using the point-norm continuity of $\alpha_\infty$ on $A_{\infty,\alpha}$, we may assume that
\[
\| \alpha_{\infty,g}(x) - \alpha_{\infty,g_0}(x) \| \leq \delta/3 \quad \text{for all}~ g\in U_0.
\]
By compactness of $K$, we find some $N \in \IN$, $\set{g_j}_{j=1}^N \subset K$, an open covering $K \subset \bigcup_{j=1}^N U_j$ such that $g_j \in U_j$ and
\[
\sup_{k \in \IN} \| \alpha_g(x_k) - \alpha_{g_j}(x_k)\| \leq \delta/3 \quad\text{and}\quad \|\alpha_{\infty,g}(x) - \alpha_{\infty,g_j}(x) \| \leq \delta/3
\]
for all $j = 1,\ldots, N$ and $g \in U_j$. Moreover, we find $n_0 \in \IN$ such that for $j = 1,\ldots,N$,
\[
\sup\limits_{k \geq n_0} \| \alpha_{g_j}(x_k) - x_k \| \leq \| \alpha_{\infty,g_j}(x) - x \| + \delta/3.
\]
We then compute for $g\in U_j$
\[
\begin{array}{lcl}
\sup\limits_{k \geq n_0} \| \alpha_g(x_k) - x_k \| & \leq &
 \sup\limits_{k \geq n_0}\| \alpha_g(x_k) - \alpha_{g_j}(x_k) \| + \sup\limits_{k\geq n_0} \| \alpha_{g_j}(x_k) - x_k \| \\
 & \leq & \delta/3 + \| \alpha_{\infty,g_j}(x) - x \| + \delta/3 \\
 & \leq & \| \alpha_{\infty,g_j}(x) - \alpha_{\infty,g}(x) \| + \| \alpha_{\infty,g}(x) - x \| + 2\delta/3 \\
 & \leq & \| \alpha_{\infty,g}(x) - x \| + \delta.
\end{array}
\]
Since the sets $U_j$ formed an open cover of $K$, this concludes the proof.
\end{proof}

Using this simple technical tool, we can generalize most of the basic properties of sequentially split $*$-homomorphisms, which were shown in the second section, to the equivariant context. Since doing this is very routine, we recommend the reader to skip the next few technical statements upon first reading, and jump right ahead to \ref{crossed products sequential dominance}, where we start discussing properties that are exclusively interesting in the equivariant context.

\begin{lemma} \label{weak-eq-seq-split}
Let $A$ and $B$ be two separable \cstar-algebras, $G$ a second countable, locally compact group and $\alpha:G\curvearrowright A$ and $\beta:G\curvearrowright B$ two continuous actions. An equivariant $*$-homomorphism $\phi:(A,\alpha)\to (B,\beta)$ is sequentially split if and only if there exists a commutative diagram of equivariant $*$-homomorphisms
\[
\xymatrix{
(A,\alpha) \ar[rr] \ar[dr]_\phi & & ((A_{\infty,\alpha})_{\infty,\alpha_\infty},(\alpha_\infty)_\infty) \\
 & (B,\beta) \ar[ru]
}
\]
\end{lemma}
\begin{proof}
Since the ``only if''-part is trivial, we only show the ``if''-part. For this, let $\psi :(B,\beta)\to ((A_{\infty,\alpha})_{\infty,\alpha_\infty},(\alpha_\infty)_\infty)$ be an equivariant $*$-homomorphism satisfying $\psi\circ\phi(a)=a$ for all $a\in A$. Find $*$-linear maps $\psi_m:B \to A_{\infty,\alpha}$, $m \in \IN$, such that $\psi(b) = [(\psi_m(b))_m]$ for all $b \in B$. Find $*$-linear maps $\psi_{(m,n)}:B\to A$, $m,n \in \IN$, such that $\psi_m(b) = [(\psi_{(m,n)}(b))_n]$ for all $b \in B$. Observe that $\psi(b)=[(\psi_{(m,n)}(b))_{(m,n)}] \in (A_{\infty,\alpha})_{\infty,\alpha_\infty}$ for all $b\in B$.

Let $H_1 \subset H_2 \subset \ldots \fin G$ be an increasing sequence of finite subsets $H_k$ such that $G' = \bigcup_{k \in \IN} H_k$ is a dense subgroup of $G$. Let $K \subset G$ be a compact neighbourhood of $\eins_G$. Let $S_1 \subset S_2 \subset \ldots \fin A $ be an increasing sequence of finite sets $S_k$ such that $A' = \bigcup_{k \in \IN} S_k$ is dense in $A$. Let $F_1 \subset F_2 \subset \ldots \fin B$ be an increasing sequence of finite sets $F_k$ such that $B' = \bigcup_{k\in \IN} F_k$ is a dense $\IQ[i]$-$*$-subalgebra of $B$. We may assume that $\phi(A') \subset B'$ and $\beta_g(B') = B'$ for all $g \in G'$.

As $\psi:(B,\beta) \to ((A_{\infty,\alpha})_{\infty,\alpha_\infty},(\alpha_\infty)_\infty)$ is an equivariant $*$-homomorphism, we find a sequence of natural numbers $(m_k)_k$ such that for $k\in \IN$, $h \in H_k$ and $b,b'\in F_k$,
\[
\limsup\limits_{n \to \infty} \| \psi_{(m_k,n)}(bb') - \psi_{(m_k,n)}(b)\psi_{(m_k,n)}(b')\| \leq 1/k
\]
and
\[
\limsup\limits_{n \to \infty} \| \psi_{(m_k,n)}(b) \| \leq \| b \| + 1/k
\]
and
\[
\limsup\limits_{n \to \infty} \| \alpha_h \circ \psi_{(m_k,n)}(b) - \psi_{(m_k,n)} \circ \beta_h(b) \| \leq 1/k.
\]
Using furthermore that $\phi \circ \psi$ coincides with the standard embedding $A \into (A_{\infty,\alpha})_{\infty,\alpha_\infty}$, we may also assume that for all $k \in \IN$ and $a \in S_k$,
\[
\limsup_{n \to \infty} \| \psi_{(m_k,n)} \circ \phi(a) - a \| \leq 1/k.
\]
Applying  \ref{uniform norm}, we may assume that for all $k \in \IN$, $b\in F_k$ and $g \in K$, we have
\[
\begin{array}{ccl}
\| \alpha_{\infty,g} \circ \psi_{m_k}(b) - \psi_{m_k}(b) \| &\leq& \| (\alpha_\infty)_{\infty,g} \circ \psi(b) - \psi(b) \| + 1/k \\
& \leq & \| \beta_g(b) - b \| + 1/k.
\end{array}
\]
Similarly, we find a sequence of natural numbers $(n_k)_k$ such that for $k \in \IN$, $h \in H_k$ and $b,b' \in F_k$, we have
\[
\| \psi_{(m_k,n_k)}(bb') - \psi_{(m_k,n_k)}(b)\psi_{(m_k,n_k)}(b')\| \leq 2/k
\]
and
\[
\| \psi_{(m_k,n_k)}(b) \| \leq \| b \| + 2/k
\]
and
\[
\| \alpha_h \circ \psi_{(m_k,n_k)}(b) - \psi_{(m_k,n_k)} \circ \beta_h(b) \| \leq 2/k.
\]
Moreover, we may assume that for all $k \in \IN$ and $a \in S_k$,
\[
\| \psi_{(m_k,n_k)} \circ \phi(a) - a \| \leq 2/k.
\]
Using that for $k \in \IN$ and $b \in B$, $\psi_{m_k}(b) = [(\psi_{(m_k,n)}(b))_n] \in A_{\infty,\alpha}$, we may by \ref{uniform norm} also assume that for $g \in K$,
\[
\begin{array}{rl}
\multicolumn{2}{l}{ \| \alpha_g \circ \psi_{(m_k,n_k)}(b) - \psi_{(m_k,n_k)}(b) \| } \\ 
\hspace{10mm}\leq& \| \alpha_{\infty,g}([(\psi_{(m_k,n)}(b))_n]) - [(\psi_{(m_k,n)}(b))_n] \| + 1/k \\
=& \| \alpha_{\infty,g}(\psi_{m_k}(b)) - \psi_{m_k}(b) \| + 1/k \\
 \leq& \| \beta_g(b) - b \| + 2/k.
\end{array}
\]
Define the $\IQ[i]$-$*$-linear map
\[
\psi':B' \to A_\infty,\ \psi'(b) = [(\psi_{(m_k,n_k)}(b))_k],
\]
that by construction is a contractive $*$-homomorphism satisfying $\psi' \circ \phi(a) = a$ for all $a\in A'$. As $\psi'$ is contractive, it extends uniquely to a $*$-homomorphism $\psi':B \to A_\infty$. Using that $A' \subset A$ is dense, we conclude that $\psi' \circ \phi$ coincides with the canonical embedding $A \into A_\infty$.

It remains to show that $\psi'$ is $\beta$-to-$\alpha_\infty$-equivariant. It follows from the construction of $\psi'$ that $\alpha_{\infty,g} \circ \psi'(b) = \psi' \circ \beta_g (b)$ for all $b \in B'$ and $g \in G'$. As $B' \subset B$ is dense, we conclude that $\psi'$ is equivariant with respect to the induced $G'$-actions. We claim that $\psi'(B) \subset A_{\infty,\alpha}$. For $b \in B'$ and $g \in K$,
\[
\begin{array}{lcl}
\| \alpha_{\infty,g} \circ \psi'(b) - \psi'(b) \| & = &
\limsup\limits_{k \to \infty} \| \alpha_g \circ \psi_{(m_k,n_k)} (b) - \psi_{(m_k,n_k)}(b) \| \\
 & \leq & \limsup\limits_{k\to \infty} \| \beta_g(b) - b \| + 2/k \\
 & = & \| \beta_g(b) - b \|.
\end{array}
\]
As $\beta$ is a continuous action, we find for given $\delta > 0$ an open neighbourhood $U \subset K$ of $1_G$ such that $\| \beta_g(b) - b \| \leq \delta$ for all $g \in U$. This shows that for all $g\in U$, $\| \alpha_{\infty,g} \circ \psi'(b) - \psi'(b) \| \leq \delta$. Hence, $\psi'(b) \in A_{\infty,\alpha}$. As $B' \subset B$ is dense, we conclude that $\psi'$ indeed maps into $A_{\infty,\alpha}$. Now let $b \in B$, $g \in G$ and find a sequence $(g_n)_n \subset G'$ that converges to $g$. Using that $\psi'$ is equivariant with respect to the $G'$-actions and that $\psi'(B) \subset A_{\infty,\alpha}$, we get that
\[
\alpha_{\infty,g}(\psi'(b)) = \lim\limits_{n \to \infty} \alpha_{\infty,g_n}(\psi'(b)) = \lim\limits_{n\to \infty} \psi'(\beta_{g_n}(b)) = \psi'(\beta_g(b)).
\]
This shows that $\phi:(A,\alpha) \to (B,\beta)$ is sequentially split.
\end{proof}

As in the non-equivariant case, one can use \ref{weak-eq-seq-split} in order to conclude that compositions of equivariantly sequentially split $*$-homomorphisms are equivariantly sequentially split. We omit the proof as it is completely analogous.

\begin{prop}
Restricted to separable \cstar-algebras with point-norm continuous actions by second countable, locally compact groups, the composition of two equivariantly sequentially split $*$-homomorphisms is equivariantly sequentially split.
\end{prop}

The following is completely analogous to \ref{sequential dominance properties 1}\ref{sdp1:1}.

\begin{prop} \label{eq-seq-split-hered}
Let $A$ and $B$ be \cstar-algebras, $G$ a locally compact group, and $\alpha:G\curvearrowright A$ and $\beta:G \curvearrowright B$ continuous actions. Let $E\subset A$ be a hereditary, $\alpha$-invariant \cstar-subalgebra. Assume that $\phi:(A,\alpha) \to (B,\beta)$ is a sequentially split $*$-homomorphism. Set $E_\phi = \overline{\phi(E)B\phi(E)}$. Then the induced equivariant $*$-homomorphism $\phi:(E,\alpha|_E)\to (E_\phi,\beta|_{E_\phi})$ is sequentially split.
\end{prop}

Next comes the equivariant analogue of \ref{inductive limits}. Since its proof is analogous, we omit it.

\begin{cor} \label{equi inductive limits}
Let $G$ be a second countable, locally compact group. Let $\set{(A_n,\alpha_n), \kappa_n}_{n\in\IN}$ and $\set{(B_n,\beta_n), \theta_n}_{n\in\IN}$ be two inductive systems of separable $G$-\cstar-algebras with ($G$-equivariant) limits $(A,\alpha)$ and $(B,\beta)$, respectively. Let $\phi_n: (A_n,\alpha_n)\to (B_n,\beta_n)$ be a sequence of equivariant $*$-homomorphisms compatible with the two inductive systems, i.e.~$\theta_n\circ\phi_n = \phi_{n+1}\circ\kappa_n$ for all $n$. Denote by $\phi: (A,\alpha)\to (B,\beta)$ the induced $*$-homomorphism given by the universal property of the inductive limits.
If each of the $*$-homomorphisms $\phi_n$ is equivariantly sequentially split, then $\phi$ is equivariantly sequentially split.
\end{cor}

\begin{rem}
The notion of equivariantly sequentially split homomorphisms can be extended in a straightforward way to \cstar-algebras equipped with endomorphic actions by semigroups. More precisely, given \cstar-algebras $A$ and $B$, a discrete semigroup $P$, and actions $\alpha:P\curvearrowright A$ and $\beta:P\curvearrowright B$ by endomorphisms, an equivariant $*$-homomorphism $\phi:(A,\alpha)\to (B,\beta)$ is called (equivariantly) sequentially split, if there exists an equivariant $*$-homomorphism $\psi:(B,\beta) \to (A_\infty,\alpha_\infty)$ making the following diagram commute
\[
\xymatrix{
(A,\alpha) \ar[rr] \ar[dr]_\phi & & (A_\infty,\alpha_\infty) \\
 & (B,\beta) \ar[ru]_\psi
}
\]
At least when $A$, $B$ and $\phi$ are unital, this notion behaves well with respect to the corresponding semigroup crossed products, that is, the \cstar-algebras that are universal for unital covariant pairs; see \cite{LacaRaeburn1996} for a precise definition. In this situation, many of the facts proved in this section hold as well. This is used in \cite[Section 5]{BarlakOmlandStammeier16} as an important tool related to the computation of the $K$-groups of \cstar-algebras arising from integral dynamics. However, we will not pursue this type of generalization in this paper.
\end{rem}

From now on, we start discussing properties of sequentially split $*$-homo\-m\-orphisms that are exclusively interesting in the equivariant context.
Namely, they turn out to enjoy the following crucial functoriality property with respect to formation of crossed products.

\begin{prop} \label{crossed products sequential dominance}
Let $A$ and $B$ be \cstar-algebras, $G$ a locally compact group and $\alpha:G\curvearrowright A$ and $\beta:G\curvearrowright B$ continuous actions. Let $\phi:(A,\alpha)\to (B,\beta)$ be an equivariantly sequentially split $*$-homomorphism. Then: 
\begin{enumerate}[label=\textup{(\roman*)},leftmargin=*]
\item The induced $*$-homomorphism $\phi\rtimes G:A\rtimes_\alpha G\to B\rtimes_\beta G$ between the crossed products is sequentially split. 
\item If $G$ is abelian, then the dual morphism $\hat{\phi}: (A\rtimes_\alpha G,\hat{\alpha})\to (B\rtimes_\beta G,\hat{\beta})$ is ($\hat{G}$-)equivariantly sequentially split.
\end{enumerate}
\end{prop}
\begin{proof}
Let $\psi:(B,\beta)\to (A_{\infty,\alpha},\alpha_\infty)$ be an equivariant approximate left-inverse for $\phi$. We obtain a commutative diagram
\[
\xymatrix{
A\rtimes_\alpha G \ar[rr] \ar[rd]_{\phi\rtimes G} & & A_{\infty,\alpha}\rtimes_{\alpha_\infty} G \\
 & B\rtimes_\beta G \ar[ru]_{\psi\rtimes G}
}
\]
There is a canonical $*$-homomorphism $A_{\infty,\alpha}\rtimes_{\alpha_\infty} G\to (A\rtimes_\alpha G)_\infty$ extending the canonical inclusion of $A\rtimes_\alpha G$ into its sequence algebra. This shows that $\phi\rtimes G:A\rtimes_\alpha G\to B\rtimes_\beta G$ is sequentially split. If $G$ is abelian, $\psi\rtimes G$ and $\phi\rtimes G$ are equivariant with respect to the respective dual actions of $\hat{G}$. Moreover, the canonical $*$-homomorphism $A_{\infty,\alpha}\rtimes_{\alpha_\infty} G\to (A\rtimes_\alpha G)_\infty$ is clearly $\widehat{\alpha_\infty}$-to-$\hat{\alpha}_\infty$-equivariant. Thus, $\hat{\phi}:(A\rtimes_\alpha G,\hat{\alpha})\to (B\rtimes_\beta G,\hat{\beta})$ is sequentially split.
\end{proof}

If the acting group is compact, a similar functoriality also applies to the fixed point algebra. For the proof, we need the following fact.

\begin{lemma} \label{compact fixed point sequence algebra}
Let $A$ be a \cstar-algebra, $G$ a compact group and $\alpha: G\curvearrowright A$ a continuous action. Then the canonical embedding from $(A^\alpha)_\infty$ into $(A_\infty)^{\alpha_\infty}$ is an isomorphism.
\end{lemma}
\begin{proof}
Let $x\in (A_\infty)^{\alpha_\infty}$ be represented by a bounded sequence $(x_n)_n$ in $A$. Since in particular $x\in A_{\infty,\alpha}$, it follows from \ref{uniform norm} that 
\[
\max_{g\in G} \|x_k-\alpha_g(x_k)\| \stackrel{k\to\infty}{\longrightarrow} 0.
\]
Let $\mu$ be the normalized Haar measure on $G$. In particular, we have 
\[
\|x_k-\int_G \alpha_g(x_k)~d\mu(g)\| \stackrel{k\to\infty}{\longrightarrow} 0,
\]
showing that $x$ can be represented by a sequence in $A^\alpha$.
\end{proof}

\begin{prop} \label{fixed point sequential dominance}
Let $A$ and $B$ be \cstar-algebras, $G$ a compact group and $\alpha:G\curvearrowright A$ and $\beta:G\curvearrowright B$ continuous actions. Assume that $\phi:(A,\alpha)\to (B,\beta)$ is an equivariantly sequentially split $*$-homomorphism. Then the induced $*$-homomorphism $\phi:A^\alpha\to B^\beta$ is sequentially split.
\end{prop}
\begin{proof}
Let $\psi:(B,\beta)\to (A_{\infty,\alpha},\alpha_\infty)$ be an equivariant approximate left-inverse for $\phi$. Passing to fixed point algebras, we obtain a commutative diagram
\[
\xymatrix{
A^\alpha \ar[rr] \ar[rd]_\phi & & (A_\infty)^{\alpha_\infty} \\
 & B^\beta \ar[ru]_\psi
}
\]
Observe that we have used the equality $(A_{\infty,\alpha})^{\alpha_\infty}=(A_\infty)^{\alpha_\infty}$. As $G$ is compact, \ref{compact fixed point sequence algebra} yields that $(A_\infty)^{\alpha_\infty}=(A^\alpha)_\infty$. This shows that $\phi:A^\alpha\to B^\beta$ is sequentially split.
\end{proof}

\begin{prop} \label{seq split stability}
Let $A$ and $B$ be \cstar-algebras, $G$ a locally compact group, and $\alpha:G\curvearrowright A$ and $\beta:G\curvearrowright B$ continuous actions. Assume that $\phi:(A,\alpha)\to (B,\beta)$ is a non-degenerate $*$-homomorphism and let $\gamma:G\curvearrowright \CK$ be any continuous action on the compact operators on some separable Hilbert space. Then $\phi$ is equivariantly sequentially split if and only if
\[
\phi\otimes \id_{\CK}:(A\otimes \CK,\alpha\otimes \gamma)\to (B\otimes \CK,\beta\otimes \gamma)
\]
is equivariantly sequentially split.
\end{prop}
\begin{proof}
Assume that $\phi:(A,\alpha)\to (B,\beta)$ is sequentially split and find an equivariant approximate left-inverse $\psi:(B, \beta) \to (A_{\infty,\alpha}, \alpha_\infty)$. The composition of $\psi \otimes \id_\CK: (B \otimes \CK,\beta \otimes \gamma) \to (A_{\infty,\alpha} \otimes \CK, \alpha_\infty \otimes \gamma)$ with the canonical $*$-homomorphism $(A_{\infty,\alpha} \otimes \CK,\alpha_\infty \otimes \gamma) \to ((A\otimes \CK)_{\infty,\alpha \otimes \gamma},(\alpha \otimes \gamma)_\infty)$ defines an equivariant approximate left-inverse for $\phi \otimes \id_\CK$. Hence $\phi\otimes \id_{\CK}:(A\otimes \CK,\alpha\otimes \gamma)\to (B\otimes \CK,\beta\otimes \gamma)$ is sequentially split. 

Assume that $\phi\otimes \id_{\CK}:(A\otimes \CK,\alpha\otimes \gamma)\to (B\otimes \CK,\beta\otimes \gamma)$ is sequentially split and find an equivariant approximate left-inverse $\psi:(B\otimes \CK,\beta\otimes\gamma)\to ((A\otimes\CK)_{\infty,\alpha\otimes\gamma},(\alpha\otimes\gamma)_\infty)$. As the canonical inclusion $A\otimes \CK\into D_{\infty,A\otimes \CK}$ and $\phi \otimes \id_\CK$ are non-degenerate, we conclude that $\psi(B)$ is a non-degenerate \cstar-subalgebra of $D_{\infty,A\otimes \CK}$. In particular, $\psi$ extends to a strictly continuous $*$-homomorphism $\psi:\CM(B\otimes \CK)\to \CM(D_{\infty,A\otimes \CK})$. Moreover, thinking of $A\otimes \CK$ as a non-degenerate \cstar-subalgebra of $\CM(D_{\infty,A\otimes\CK})$, we also may consider $\eins \otimes \CK$ as a non-degenerate \cstar-subalgebra of $\CM(D_{\infty,A\otimes\CK})$.

Restricting $\psi:\CM(B\otimes \CK)\to \CM(D_{\infty,A\otimes \CK})$ to $B\cong B\otimes\eins \subset \CM(B\otimes \CK) \cap (\eins \otimes \CK)'$, we obtain an equivariant $*$-homomorphism
\[
\psi':(B,\beta)\to (\CM_{\alpha\otimes \gamma}(D_{\infty,A\otimes \CK})\cap (\eins \otimes \CK)',(\widetilde{\alpha\otimes \gamma})_\infty).
\]
By \ref{relative commutant compact},
\[
\CM_{\alpha\otimes \gamma}(D_{\infty,A\otimes\CK})\cap (\eins \otimes \CK)' = \CM_\alpha(D_{\infty,A})\otimes \eins,
\]
and we therefore obtain a commutative diagram
\[
\xymatrix{
 (A,\alpha) \ar[rd]_\phi \ar[rr] & &  (\CM_\alpha (D_{\infty,A}),\tilde{\alpha}_\infty) \\
 & (B,\beta) \ar[ru]_{\psi'}
}
\]
Using that $\phi$ is non-degenerate, we conclude that $\psi'(B) \subset D_{\infty,A}$. We can thus consider $\psi'$ as an equivariant $*$-homomorphism $\psi':(B,\beta) \to (A_{\infty,\alpha},\alpha_\infty)$. This shows that $\phi:(A,\alpha)\to (B,\beta)$ is sequentially split.
\end{proof}

\begin{rem}
Let $A$ be a \cstar-algebra, $G$ a locally compact, abelian group and $\alpha:G\curvearrowright A$ a continuous action. By the Takai duality theorem \cite{Takai1975}, it is well-known that $(A\rtimes_\alpha G\rtimes_{\hat{\alpha}} \hat{G}, \hat{\hat{\alpha}})$ is conjugate to $(A\otimes \CK(L^2(G)), \alpha\otimes\rho)$, where $\rho$ is the $G$-action on $\CK(L^2(G))$ induced by the right-regular representation. Moreover, this isomorphism is natural in $(A,\alpha)$. In particular, this means that for any equivariant $*$-homomorphism $\phi: (A,\alpha)\to (B,\beta)$, there are equivariant isomorphisms 
\[
\kappa_A: (A\rtimes_\alpha G\rtimes_{\hat{\alpha}} \hat{G}, \hat{\hat{\alpha}}) \stackrel{\cong}{\longrightarrow} (A\otimes\CK(L^2(G)), \alpha\otimes\rho)
\]
and
\[
\kappa_B: (B\rtimes_\beta G\rtimes_{\hat{\beta}} \hat{G}, \hat{\hat{\beta}}) \stackrel{\cong}{\longrightarrow} (B\otimes\CK(L^2(G)), \beta\otimes\rho)
\]
such that the following diagram is commutative:
\[
\xymatrix{
(A\rtimes_\alpha G\rtimes_{\hat{\alpha}} \hat{G}, \hat{\hat{\alpha}}) \ar[rr]^{\hat{\hat{\phi}}} \ar[d]_{\kappa_A} && (B\rtimes_\beta G\rtimes_{\hat{\beta}} \hat{G}, \hat{\hat{\beta}}) \ar[d]^{\kappa_B} \\
(A\otimes\CK(L^2(G)), \alpha\otimes\rho) \ar[rr]^{\phi\otimes\id_{\CK(L^2(G))}} && (B\otimes\CK(L^2(G)), \beta\otimes\rho)
}
\]
\end{rem}

Combining this with \ref{seq split stability}, we obtain the following:

\begin{cor} \label{ddual}
Let $A$ and $B$ be \cstar-algebras, $G$ a second countable, locally compact, abelian group, and $\alpha:G\curvearrowright A$ and $\beta:G \curvearrowright B$ continuous actions. Let $\phi:(A,\alpha) \to (B,\beta)$ be an equivariant $*$-homomorphism. Then $\phi$ is equivariantly sequentially split if and only if its double dual morphism
\[
\hat{\hat{\phi}}: (A\rtimes_\alpha G\rtimes_{\hat{\alpha}} \hat{G}, \hat{\hat{\alpha}})\to (B\rtimes_\beta G\rtimes_{\hat{\beta}} \hat{G}, \hat{\hat{\beta}})
\] 
is equivariantly sequentially split.
\end{cor}

This, in turn, immediately implies the following duality:

\begin{cor} \label{seq split duality}
Let $A$ and $B$ be \cstar-algebras, $G$ a second countable, locally compact, abelian group, and $\alpha:G\curvearrowright A$ and $\beta:G \curvearrowright B$ continuous actions. Let $\phi:(A,\alpha) \to (B,\beta)$ be an equivariant $*$-homomorphism. Then $\phi$ is equivariantly sequentially split if and only if its dual $*$-homomorphism $\hat{\phi}: (A\rtimes_\alpha G, \hat{\alpha})\to (B\rtimes_\beta G, \hat{\beta})$ is ($\hat{G}$-)equivariantly sequentially split.
\end{cor}
\begin{proof}
This follows from \ref{crossed products sequential dominance} and \ref{ddual}.
\end{proof}


\section{Applications}

\subsection{Rokhlin actions of compact groups}

\begin{defi}[cf.\ \cite{HirshbergWinter2007}]
Let $A$ be a separable \cstar-algebra and $G$ a second countable, compact group. Let $\sigma:G\curvearrowright \CC(G)$ denote the canonical $G$-shift, that is, $\sigma_g(f)=f(g^{-1}\cdot\_)$ for all $f \in \CC(G)$ and $g\in G$. A continuous action $\alpha:G\curvearrowright A$ is said to have the Rokhlin property if there exists a unital and equivariant $*$-homomorphism
\[
(\CC(G),\sigma)\to (F_{\infty,\alpha}(A),\tilde{\alpha}_\infty).
\] 
\end{defi}

The Rokhlin property turns out to fit formidably into the concept of sequentially split $*$-homomorphisms. This will be a consequence of the following equivariant generalization of \ref{unital homs in FA}. The proof is a straightforward generalization from the non-equivariant case.

\begin{lemma}
\label{eq unital homs FA}
Let $A$ be a $\sigma$-unital \cstar-algebra and $C$ a unital \cstar-algebra. Let $G$ be a second-countable, locally compact group, and let $\alpha: G\curvearrowright A$ and $\gamma: G\curvearrowright C$ be continuous actions. There exists an equivariant and unital $*$-homomorphism from $(C,\gamma)$ to $\big( F_{\infty,\alpha}(A), \tilde{\alpha}_\infty \big)$ if and only if the first-factor embedding $\id_A\otimes\eins: (A,\alpha)\to (A\otimes_{\max} C, \alpha\otimes\gamma)$ is equivariantly sequentially split.
\end{lemma}

\begin{cor} \label{seq split Rp}
Let $A$ be a separable \cstar-algebra, $G$ be a second-countable, compact group and $\alpha:G\curvearrowright A$ a continuous action. Then $\alpha$ has the Rokhlin property if and only if the second-factor embedding
\[
\eins\otimes \id_A:(A,\alpha)\into (\CC(G)\otimes A,\sigma\otimes\alpha)
\]
is sequentially split.
\end{cor}

For the rest of this subsection, we will use this observation to provide a conceptual proof of the fact that crossed product \cstar-algebras by Rokhlin actions of compact groups inherit many properties from the coefficient \cstar-algebra. But first, we need some preparation.

\begin{nota} Let $A$ be a \cstar-algebra, $G$ a compact group and $\alpha: G\to\Aut(A)$ a continuous action. We will denote by $\bar{\alpha}\in\Aut(\CC(G)\otimes A)$ the induced automorphism given by
\[
\bar{\alpha}(f)(g) = \alpha_g(f(g))\quad\text{for all}~g\in G~\text{and}~f\in\CC(G,A).
\]
Moreover, we will denote by $\alpha_\mathrm{co}: A\to \CC(G)\otimes A$ the corresponding coaction of $\CC(G)$, viewed as a Hopf-\cstar-algebra, on $A$. That is, $\alpha_\mathrm{co}(a)(g)=\alpha_g(a)$ for all $g \in G$ and $a \in A$.
\end{nota}

The following is a well-known fact:

\begin{prop} \label{shift absorption}
Let $A$ be a \cstar-algebra, $G$ a compact group and $\alpha: G\curvearrowright A$  continuous action. Then the \cstar-dynamical system $(\CC(G)\otimes A,\sigma\otimes\id)$ is conjugate to $(\CC(G)\otimes A, \sigma\otimes\alpha)$ via $\bar{\alpha}$. 
In particular, $(\CC(G)\otimes A)\rtimes_{\sigma\otimes\alpha} G$ is always isomorphic to $A\otimes\CK(L^2(G))$.
\end{prop}

We note that variants of the following statement have been observed by Gardella in \cite[Section 2]{Gardella2014_R} and by the second author in the proofs of \cite[2.5, 2.6]{Szabo2015}.

\begin{theorem} \label{Rokhlin fixed point seq split}
Let $A$ be a separable \cstar-algebra, $G$ a second countable, compact group and $\alpha: G\curvearrowright A$ a continuous action. Assume that $\alpha$ has the Rokhlin property. Then 
\begin{enumerate}[label=\textup{(\roman*)},leftmargin=*]
\item the inclusion map $A^\alpha\into A$ is sequentially split. \label{rpsq:1}
\item The canonical embedding $A\rtimes_\alpha G\into A\otimes \CK(L^2(G))$, given by the isomorphism $A\otimes\CK(L^2(G)) \cong (\CC(G)\otimes A)\rtimes_{\sigma\otimes\alpha} G$ and induced by the equivariant second-factor embedding $A\into\CC(G)\otimes A$, is sequentially split. \label{rpsq:2}
\end{enumerate}
\end{theorem}
\begin{proof}
\ref{rpsq:1}: Since $\alpha$ has the Rokhlin property,  \ref{seq split Rp} implies that
\[
\eins\otimes \id_A:(A,\alpha)\into (\CC(G)\otimes A,\sigma\otimes \alpha)
\] 
is sequentially split. As $G$ is compact, we can apply \ref{fixed point sequential dominance} and conclude that
\[
\eins\otimes \id_A: A^\alpha\into (\CC(G)\otimes A)^{\sigma\otimes \alpha}
\]
is also sequentially split. One easily checks that a function $f\in \CC(G,A)$ is fixed under $\sigma \otimes \alpha$ if and only if there is some $a\in A$ such that $f(g) = \alpha_g(a)$ for all $g \in G$. In particular, $\alpha_{\mathrm{co}}: A \to (\CC(G) \otimes A)^{\sigma \otimes \alpha}$ is an isomorphism. Moreover the following diagram commutes:
\[
\xymatrix@C+0.3cm{
A^\alpha \ar@^{(->}[r] \ar@^{(->}[d] & A \ar[d]^{\alpha_{\mathrm{co}}}_\cong \\
A \ar@^{(->}[r]^/-0.5cm/{\eins\otimes\id_A} & (\CC(G)\otimes A)^{\sigma\otimes \alpha}
}
\]
This shows that the canonical embedding $A^\alpha\into A$ is sequentially split.

The second statement \ref{rpsq:2} follows directly from \ref{crossed products sequential dominance}, \ref{seq split Rp} and \ref{shift absorption}.
\end{proof}

The following result arises as an immediate consequence, and generalizes many permanence property results of \cite{OsakaPhillips2012, Santiago2014, HirshbergWinter2007, Gardella2014_R}. The statement about the UCT is a significant improvement of the main result of \cite{Szabo2015}.

\begin{cor}
Let $A$ be a separable \cstar-algebra, $G$ a second-countable, compact group and $\alpha: G\curvearrowright A$ a continuous action with the Rokhlin property. 
Then all the properties listed in \ref{sequential dominance properties 2} pass from $A$ to $A^\alpha$ and from $A\otimes\CK(L^2(G))$ to $A\rtimes_\alpha G$. Moreover, if $A$ is nuclear and satisfies the UCT, then so do $A^\alpha$ and $A\rtimes_\alpha G$.
\end{cor}

\begin{rem}
Another consequence of \ref{Rokhlin fixed point seq split} (together with \ref{sequential dominance properties 1}\ref{sdp1:6}) is that the canonical inclusion $A^\alpha\into A$ is injective in $K$-theory, whenever $\alpha$ has the Rokhlin property. In the case that $G$ is finite and the \cstar-algebra $A$ is unital and simple, this was shown by Izumi \cite[3.13]{Izumi2004}. Izumi moreover proved that for Rokhlin actions of finite groups, the image $K_*(A^\alpha)\to K_*(A)$ coincides with the subgroup of fixed points of the induced action on the $K$-theory group $K_*(A)$. This striking result allows one, in contrast to many situations where the Rokhlin property is absent, to determine the $K$-theory of the crossed product in a very straightforward manner. Using the language of sequentially split $*$-homomorphisms, we will now see that this generalizes to the case of compact group actions with the Rokhlin property on separable \cstar-algebras.
\end{rem}

\begin{theorem} \label{Rokhlin K-theory}
Let $A$ be a separable \cstar-algebra, $G$ a second countable, compact group and $\alpha:G\curvearrowright A$ a continuous action. If $\alpha$ has the Rokhlin property, then
\[
\im(K_*(A^\alpha)\into K_*(A)) = \set{x\in K_*(A) ~|~ K_*(\alpha_\mathrm{co})(x) = K_*(\eins_{\CC(G)}\otimes \id_A)(x)}.
\]
The analogous statement is true for $K$-theory with coefficients.
\end{theorem}
\begin{proof}
Clearly, every $x\in \im(K_*(A^\alpha)\into K_*(A))$ satisfies $K_*(\alpha_\mathrm{co})(x)=K_*(\eins \otimes \id_A)(x)$. For the other inclusion, let $x\in K_0(A)$ satisfy 
\[
K_0(\alpha_\mathrm{co})(x) = K_0(\eins\otimes \id_A)(x)\in K_0(\CC(G)\otimes A).
\]
Find some $n\in \IN$ and a projection $p\in M_n(\tilde{A})$ such that 
\[
x=[p]_0 - [\eps(p)]_0\in K_0(A).
\]
Here, $\eps: \tilde{A}\to \IC$ is the canonical character, which we also view as extended to $\eps: (\CC(G)\otimes A)^\sim\to\IC$. Observe that we denote the matrix amplification of a $*$-homomorphisms again by the same symbol. We have
\[
[\tilde{\alpha}_\mathrm{co}(p)] - [\eps(\tilde{\alpha}_\mathrm{co}(p))] = [(\eins\otimes \id_A)^\sim(p)] - [\eps((\eins\otimes \id_A)^\sim(p))] \in K_0(\CC(G)\otimes A).
\]
By definition of the $K_0$-group, we find $k,l\geq 0$ satisfying
\[
\tilde{\alpha}_\mathrm{co}(p)\oplus \eins_k\oplus 0_l \sim_{\mathrm{MvN}} (\eins\otimes \id_A)^\sim(p)\oplus \eins_{k}\oplus 0_{l} \text{ in}~ M_{n+k+l}((\CC(G)\otimes A)^\sim).
\]
Write $r=n+k+l$. As $\alpha$ has the Rokhlin property and the property of being equivariantly sequentially split passes to unitazations and matrix amplifications (endowed with the respective $G$-actions), we conclude from \ref{seq split Rp} that
\[
(\eins\otimes \id_A)^\sim : (M_{r}(\tilde{A}),\tilde{\alpha})\to (M_{r}((\CC(G)\otimes A)^\sim),(\sigma\otimes \id_A)^\sim) 
\]
is sequentially split. Let 
\[
\psi:(M_{r}((\CC(G)\otimes A)^\sim),(\sigma\otimes \id_A)^\sim)\to (M_{r}(\tilde{A})_{\infty,\tilde{\alpha}},\tilde{\alpha}_\infty)
\]
be an equivariant approximate left-inverse for $(\eins\otimes \id_A)^\sim$. We have that
\[
\psi\circ \tilde{\alpha}_\mathrm{co}(p\oplus \eins_{k} \oplus 0_{l}) \sim_{\mathrm{MvN}} p\oplus \eins_{k} \oplus 0_{l} \text{ in}~ M_{r}(\tilde{A})_\infty.
\]
As $\alpha_\mathrm{co}$ maps $A$ into $(\CC(G)\otimes A)^{\sigma\otimes \alpha}$, it follows that
\[
\tilde{\alpha}_\mathrm{co}(M_{r}(\tilde{A}))\subset M_{r}((\CC(G)\otimes A)^\sim)^{(\sigma\otimes \id_A)^\sim}.
\]
By equivariance of $\psi$ and \ref{compact fixed point sequence algebra}, we get that
\[
\psi\circ\tilde{\alpha}_\mathrm{co} (M_{r}(\tilde{A})) \subset (M_{r}(\tilde{A})_\infty)^{\tilde{\alpha}_\infty} = M_{r}((A^\alpha)^\sim)_\infty. 
\]
Since the relation of being a partial isometry with a fixed range projection is weakly stable, this shows that there exists some projection $q\in M_{r}((A^\alpha)^{\sim})$ with the property that
\[
q \sim_{\mathrm{MvN}} p\oplus \eins_{k}\oplus 0_{l} ~\text{in}~ M_{r}(\tilde{A}).
\]
By definition of $K_0(A)$, we get that
\[
x = [p] - [\eps(p)] = [q] - [\eps(q)] \in K_0(A),
\]
and we conclude that $x\in \im(K_0(A^\alpha)\into K_0(A))$.

For the assertion for $K_1$, observe first that the continuous action $S\alpha: G\curvearrowright SA$ on the suspension has the Rokhlin property. The fixed point algebra $(SA)^{S\alpha}$ equals $SA^\alpha$. Therefore, the injective homomorphism 
\[
K_0((SA)^{S\alpha})\into K_0(SA)
\] 
is nothing but $K_1(A^\alpha \into A)$. One also has $(S\alpha)_\mathrm{co} = S\alpha_\mathrm{co}$. Hence, the assertion for $K_0((SA)^{S\alpha})\into K_0(SA)$ translates to
\[
\im(K_1(A^\alpha)\into K_1(A)) = \set{x\in K_1(A) ~|~ K_1(\alpha_\mathrm{co})(x) = K_1(\eins\otimes \id_A)(x)}.
\]
This concludes the proof for $K_*$. The assertion for $K$-theory with coefficients in $\IZ_n$, $n\geq 2$, follows from the $K_0$-formula with the completely analogous argument, by tensoring with the trivial action on $\CO_{n+1}$ instead of taking the suspension.
\end{proof}


\subsection{Inclusions of \cstar-algebras with the Rokhlin property}
\text{}

\noindent
In \cite{OsakaKodakaTeruya2009}, Osaka, Kodaka and Teruya defined the Rokhlin property for an inclusion $A\subset B$ of unital \cstar-algebras.

\begin{defi}[cf.\ \cite{Watatani1990}] Let $B$ be a unital \cstar-algebra and $A\subset B$ a unital \cstar-subalgebra. Moreover, let $E: B\to A$ be a conditional expectation. Then $E$ is said to have a quasi-basis, if there exist elements $u_1,v_1,\dots,u_n,v_n\in B$ such that
\[
x= \sum_{j=1}^n u_n E(v_nx) = \sum_{j=1}^n E(xu_n)v_n\quad\text{for all}~x\in B.
\]
In this case, one defines the Watatani index of $E$ as
\[
\ind(E) = \sum_{j=1}^n u_nv_n\in B.
\]
If $A\subset B$ is some inclusion of unital \cstar-algebras such that there exists a conditional expectation $E: B\to A$ with a quasi-basis, one also says that this inclusion has finite Watatani Index.
\end{defi}

\begin{theorem}[cf.~{\cite[1.2.8, 2.1.7, 2.3.1]{Watatani1990}}] Let $A\subset B$ be an inclusion of unital \cstar-algebras and assume that $E: B\to A$ is a conditional expectation with finite Watatani index. Then the index $\ind(E)$ does not depend on the choice of the quasi-basis and is a positive, central and invertible element of $B$.
\end{theorem}

\begin{defi}[cf.~{\cite[3.1]{OsakaKodakaTeruya2009}}]
Let $A\subset B$ be an inclusion of separable, unital \cstar-algebras and assume that $E: B\to A$ is a conditional expectation with finite Watatani index. Denote by $E_\infty: B_\infty\to A_\infty$ the canonical extension of $E$ to the sequence algebras given by componentwise application of $E$. The conditional expectation $E$ is said to have the Rokhlin property, if there exists a projection $p\in B_\infty\cap B'$ with $E_\infty(p)=\ind(E)^{-1}$ and such that the map $x\mapsto px$ is injective on $B$.
\end{defi}

\begin{prop}[cf.~{\cite[3.6]{OsakaKodakaTeruya2009}}]
Let $A\subset B$ be an inclusion of separable, unital \cstar-algebras and assume that $E:A\to B$ is a conditional expectation with the Rokhlin property. Then any conditional expectation from $A$ to $B$ with finite Watatani index has the Rokhlin property. Hence, the Rokhlin property is a property of inclusions of separable, unital \cstar-algebras.
\end{prop}

\begin{rem}
It turns out that this notion is indeed a generalization of finite group actions with the Rokhlin property. It is a simple exercise to show the following, see also \cite[3.2]{OsakaKodakaTeruya2009}:

Let $\alpha: G\curvearrowright A$ be a finite group action on a separable, unital, simple \cstar-algebra. Consider the conditional expectation $E: A\to A^\alpha$ onto the fixed point algebra given by $E(a)=|G|^{-1}\sum_{g\in G}\alpha_g(a)$, which is well-known to have finite index $|G|^{-1}\eins_A$. Then the inclusion $A^\alpha\subset A$ has the Rokhlin property if and only if the action $\alpha$ has the Rokhlin property. 
\end{rem}

\begin{prop}[cf.~{\cite[5.1]{OsakaKodakaTeruya2009}} with proof] \label{Rokhlin inclusion inv}
Let $A\subset B$ be an inclusion of separable, unital \cstar-algebras. Assume that $E: B\to A$ is a conditional expectation with finite Watatani index. Denote by $E_\infty: B_\infty\to A_\infty$ the canonical extension of $E$ to the sequence algebras given by componentwise application of $E$. Assume that this inclusion has the Rokhlin property with a Rokhlin projection $p\in B_\infty\cap B'$. Then for each $x\in B_\infty$, the product $\ind(E)\cdot E_\infty(xp)$ is the unique element $y\in A_\infty$ satisfying the equation $xp=yp$.
\end{prop}

As it turns out, inclusions with the Rokhlin property are always sequentially split:

\begin{theorem} \label{Rokhlin inclusions}
Let $A\subset B$ be an inclusion of separable, unital \cstar-algebras with the Rokhlin property. Then the inclusion $*$-homomorphism is sequentially split.
\end{theorem}
\begin{proof}
Let $p\in B_\infty\cap B'$ be a Rokhlin projection. Define $\psi: B\to A_\infty$ via $\psi(x)=\ind(E)\cdot E_\infty(xp)$. Then this is obviously a u.c.p.~map. Since $p$ is in the central sequence algebra of $B$, we apply \ref{Rokhlin inclusion inv} and obtain
\[
\psi(xy)p = xyp = xpy = \psi(x)py = \psi(x)yp = \psi(x)\psi(y)p
\]
for all $x,y\in B$. By uniqueness, we obtain $\psi(xy)=\psi(x)\psi(y)$ for all $x,y\in B$, and thus $\psi$ is a unital $*$-homomorphism. Lastly, observe that \[
\psi(a)=\ind(E)E_\infty(ap) = \ind(E)E_\infty(pa) = \ind(E)E_\infty(p)a = a
\] 
for all $a\in A$. This finishes the proof. 
\end{proof}

Combining this observation with the permanence properties established in the second section, we can recover and extend the main results of \cite{OsakaKodakaTeruya2009, OsakaTeruya2011, OsakaTeruya2014}:

\begin{cor}[cf.~\cite{OsakaKodakaTeruya2009, OsakaTeruya2011, OsakaTeruya2014}]
Let $A\subset B$ be an inclusion of separable, unital \cstar-algebras with the Rokhlin property. If $B$ satisfies any of the properties listed in \ref{sequential dominance properties 2}, then so does $A$. Moreover, if $B$ is nuclear and satisfies the UCT, then so does $A$.
\end{cor}


\subsection{Existential embeddings}
\text{}

\noindent
Let us briefly recall the notion of an existential embedding, which originally stems from model theory of metric structures and was introduced for \cstar-algebras in \cite{GoldbringSinclair2015}. See also \cite{FarahHartSherman2014}.

\begin{defi}[cf.~{\cite[Section 2]{GoldbringSinclair2015}}]
Let $A$ and $B$ be \cstar-algebras with an embedding $\iota: A\to B$. Then $\iota$ is called an existential embedding, if for every quantifier-free formula $\phi(\bar{x},\bar{y})$ (for tuples of variables $\bar{x}, \bar{y}$), any $n\geq 1$ and any tuple $\bar{a}$ from $A$, we have
\[
\inf\set{ \phi(\bar{a}, \bar{x}) \mid \bar{x}\in A_{\leq n} } = \inf\set{ \phi(\iota(\bar{a}), \bar{y}) \mid \bar{y}\in B_{\leq n} }.
\]
\end{defi}

As it turns out, an existential embedding into a separable \cstar-algebra is a special case of a sequentially split $*$-homomorphism. We thank Ilijas Farah for pointing this out to us. Also compare with \cite[2.14]{GoldbringSinclair2015}.

\begin{theorem} \label{ec-seq-split}
Let $A$ and $B$ be separable \cstar-algebras with an embedding $\iota: A\to B$. If $\iota$ is an existential embedding, then $\iota$ is sequentially split.
\end{theorem}
\begin{proof}
Assume that $\iota$ is existential. Without loss of generality, assume that $A\subset B$ and $\iota$ is the inclusion map. As $A$ and $B$ are separable, we can choose countable, dense, $\IQ[i]$-$*$-subalgebras $A'\subset A$ and $B'\subset B$. Moreover, we may choose this in such a way that $A'=A\cap B'$. Let us consider the countable set of indeterminants $\set{X_b}_{b\in B'}$ indexed over $B'$. Then the expressions
\[
\begin{array}{ccl}
f_{\lambda, a}(X_b, X_{b_1}, X_{b_2}, X_a) &=& \|X_{b}^*-X_{b^*}\|+
\|\lambda X_{b_1} + X_{b_2} - X_{\lambda b_1+b_2}\| \\
&& +\|X_{b_1}\cdot X_{b_2}-X_{b_1b_2}\|+\|a-X_{a}\|
\end{array}
\]
for $\lambda\in\IQ[i]$,~$b, b_1, b_2\in B'$ and $a\in A'$, define a countable set of quantifier-free formulas with parameters in $\IQ[i]$ and $A'\subset A$. Evaluating $X_b=b$ for all $b\in B'$ yields that in $B$, the norms of the above formulas evaluate at zero simultaneously.  Note that the minimum of finitely many quantifier-free formulas is a quantifier-free formula. As $\iota$ is an existential embedding, we have for every $H\fin\IQ[i]$, $r>0$ and $F\fin B'_{\leq r}$ that
\[
\inf_{ \set{x_b}_{b\in F}\in A_{\leq r}}~ \inf_{\lambda\in H, b, b_1, b_2\in F, a\in F\cap A}~ f_{\lambda, a}(x_b, x_{b_1}, x_{b_2}, x_a) = 0.
\]
Now pick increasing finite sets $F_n\fin B'$ with $B'=\bigcup_{n\in\IN} F_n$. By applying the above condition, it follows that there exist sequences $x_b^{(n)}$ in $A$ (for $b\in B'$) satisfying
\begin{itemize}
\item $\limsup_{n\to\infty} \|x^{(n)}_b\|\leq \|b\|$
\item $\|\lambda x^{(n)}_{b_1} + x^{(n)}_{b_2} - x^{(n)}_{\lambda b_1+b_2}\|\to 0$
\item $\|x^{(n)}_{b_1}\cdot x^{(n)}_{b_2}-x^{(n)}_{b_1b_2}\|\to 0$
\item $\|x^{(n)*}_{b}-x^{(n)}_{b^*}\|\to 0$
\item $\|a-x^{(n)}_{a}\|\to 0$
\end{itemize}
for all $\lambda\in\IQ[i]$,~$b, b_1, b_2\in B'$ and $a\in A'$. These relations imply that the map $\psi: B'\to A_\infty$ given by $\psi(b)=[(x^{(n)}_b)_n]$ is a well-defined, contractive $*$-homomorphism with $\psi(\iota(a))=\psi(a)=a$ for all $a\in A'$. Since $A'\subset A$ and $B'\subset B$ are dense, it follows that there is a unique continuous extension $\psi: B\to A_\infty$, which is a  $*$-homomorphism with $\psi(\iota(a))=a$ for all $a\in A$. Thus $\psi$ is an approximate left-inverse for $\iota$.
\end{proof}

\begin{rem}
Let $\omega\in\beta\IN\setminus\IN$ be a free ultrafilter. Using some more continuous logic from \cite{FarahHartSherman2014}, one can improve \ref{ec-seq-split} and show that in fact, a $*$-homomorphism $\iota: A\to B$ between separable \cstar-algebras is an existential embedding if and only if $\iota$ has a faithful, approximate left-inverse into $A_\omega$. By separability, this is equivalent to having it into $A_\infty$, by virtue of a reindexation argument. This was pointed out to the authors by Ilijas Farah in personal communication. As this would require recalling more machinery from \cite{FarahHartSherman2014, GoldbringSinclair2015}, we omit the proof for the sake of brevity. Since the initial preprint version of this paper was available, more connections to model theory were discovered in \cite{GoldbringSinclair16, GardellaLupini16}. In particular, it turns out that a $*$-homomorphism between separable \cstar-algebras is sequentially split if and only if it is positively existential.
\end{rem}


\subsection{Approximately representable actions of discrete groups}

\begin{defi}[cf.\ {\cite[2.2]{IzumiMatui2010}} and {\cite[3.6]{Izumi2004}}] \label{defi approx repr}
Let $A$ be a separable \cstar-algebra and $H$ a discrete group. An action $\alpha:H \curvearrowright A$ is called approximately representable, if there exist contractions $x_{n,h} \in A$, $n\in \IN$ and $h\in H$, satisfying the following properties:
\begin{enumerate}[label=(\arabic*)]
\item For $h\in H$, $(x_{n,h}x^*_{n,h})_n$ and $(x^*_{n,h}x_{n,h})_n$ are approximate units for $A$.
\item For all $g,h\in H$ and $a\in A$,
\[
\lim\limits_{n\to \infty} \| a(x_{n,g}x_{n,h} - x_{n,gh}) \| + \| (x_{n,g}x_{n,h} - x_{n,gh})a \| = 0.
\]
\item For all $a\in A$ and $h\in H$, $\lim\limits_{n\to \infty}\| \alpha_h(a) - x_{n,h}ax_{n,h}^* \|=0$,
\item For all $g,h\in H$ and $a\in A$,
\[
\lim\limits_{n\to \infty} \| a(x_{n,ghg^{-1}} - \alpha_g(x_{n,h})) \| + \| (x_{n,ghg^{-1}} - \alpha_g(x_{n,h}))a \|=0.
\]
\end{enumerate} 
\end{defi}

\begin{rem}
In the unital case, one indeed recovers Izumi-Matui's defintion in \cite[2.2]{IzumiMatui2010}. To see this, observe that for $h \in H$, the approximate units $(x_{n,h}x_{n,h}^*)_n$ and $(x_{n,h}^*x_{n,h})_n$ converge to the unit of $A$. Hence, by a perturbation argument, we can choose the $x_{n,h}$ to be unitaries satisfying (2) and (4) for $a = \eins$, and (3) for all $a\in A$. Moreover, if $H$ is abelian, (4) implies that the unitaries $x_{n,h}$ are approximately fixed by $\alpha$. Using \ref{compact fixed point sequence algebra}, one can therefore recover Izumi's original definition \cite[3.6]{Izumi2004} of approximately representable finite, abelian group actions.
\end{rem}

\begin{notae}
Let $G$ be a locally compact group. The canonical unitary representation $G\to\CU(\CM(\mathrm{C}^*(G)))$ will be denoted by $g\mapsto\lambda^G_g$. If $\alpha: G\curvearrowright A$ is a continuous action on a \cstar-algebra, we denote the canonical unitary representation $G\to\CU(\CM(A\rtimes_\alpha G))$ by $g\mapsto\lambda^\alpha_g$.
\end{notae}

Like the Rokhlin property, approximate representability can be characterized in terms of sequentially split $*$-homomorphisms:

\begin{prop} \label{char appr rep}
Let $A$ be a separable \cstar-algebra and $H$ a countable, discrete group. Let $\alpha: H \curvearrowright A$ be an action.
The following are equivalent:
\begin{enumerate}[label=\textup{(\roman*)},leftmargin=*]
\item $\alpha$ is approximately representable. \label{car:1}
\item There is a unitary representation $w: H \to \CU(\CM(D_{\infty,A}))$ satisfying
\[
\alpha_h(a) = w_h a w_h^*\quad \text{and}\quad \tilde{\alpha}_{\infty,h}(w_g) = w_{hgh^{-1}}
\]
for all $a \in A$ and $g,h \in H$. \label{car:2}
\item The canonical inclusion $\iota_A: (A,\alpha)\into (A\rtimes_\alpha H,\ad(\lambda^\alpha))$ is equivariantly sequentially split. \label{car:3}
\end{enumerate}
\end{prop}
\begin{proof}
\ref{car:1}$\implies$\ref{car:2}:
Assume that $\alpha$ is approximately representable and take contractions $x_{n,h} \in A$ as in \ref{defi approx repr}. Set $x_h := [(x_{n,h})_n] \in A_\infty$. We first show that $x_h \in \CN(D_{\infty,A}, A_\infty)$. Let $e \in A$ be a strictly positive element. Using (1) and (3) of \ref{defi approx repr}, we obtain
\[
x_h e^2 = x_h e x_{n,h}^*x_{n,h} e = \alpha_h(e)x_{n,h}e \in D_{\infty,A}.
\]
A similar computation shows that $e^2 x_h = e x_h \alpha_{h^{-1}}(e) \in D_{\infty,A}$. Since $e^2$ is also strictly positive as an element in $D_{\infty,A}$, we conclude that $x_hD_{\infty,A}+D_{\infty,A}x_h\subset D_{\infty,A}$. Moreover, we have $x_h^*x_he=ex_h^*x_h=e$ and $x_hx_h^*e=ex_hx_h^*=e$ by (1), so the elements $x_h^*x_h$ and $x_hx_h^*$ act like a unit on $D_{\infty,A}$. Condition (2) means $x_gx_h-x_{gh}\in\ann(A,A_\infty)$ and condition (4) implies $\alpha_{\infty,h}(x_g)-x_{hgh^{-1}}\in\ann(A,A_\infty)$ for all $g,h\in H$. 

Now recall from \ref{surjection normalizer multiplier}\ref{snm:1} the natural (and therefore equivariant) surjection
\[
\pi: \CN(D_{\infty,A}, A_\infty) \to \CM(D_{\infty,A}).
\] 
Then by the properties of the $(x_h)_{h\in H}$, the map $w:H \to \CU(\CM(D_\infty(A)))$ given by $w_h := \pi(x_h)$ defines a unitary representation and has the desired properties.

\ref{car:2}$\implies$\ref{car:1}:
For $h\in H$, let $[(x_{n,h})_n] \in \CN(D_{\infty,A}, A_\infty)$ be a contraction that gets mapped to $w_h$ under the canonical surjective (cf.~\ref{surjection normalizer multiplier}\ref{snm:1}) $*$-homomorphism $\CN(D_{\infty,A}, A_\infty) \to \CM(D_{\infty,A})$. We may assume that each $x_{n,h}$ is a contraction. Using that this $*$-homomorphism is equivariant with respect to the induced $H$-actions, one checks that the $x_{n,h}$, $n\in \IN$ and $h\in H$, must satisfy the conditions (1)-(4) from \ref{defi approx repr}.

\ref{car:2}$\implies$\ref{car:3}:
The canonical embedding $A\into D_{\infty,A}$ and $w$ define a covariant pair for $(A,\alpha)$, which in turn gives rise to a $*$-homomorphism
\[
\psi:A\rtimes_\alpha H\to D_{\infty,A}\subset A_\infty.
\]
By design, $\psi\circ \iota_A(a) = a$ for all $a \in A$. Since $\tilde{\alpha}_{\infty,h}(w_g)=w_{hgh^{-1}}$ for all $g,h\in H$, one concludes that $\psi$ is $\ad(\lambda^\alpha)$-to-$\alpha_\infty$-equivariant. This shows that $\iota_A:(A,\alpha) \into (A \rtimes_\alpha H,\ad(\lambda^\alpha))$ is sequentially split.

\ref{car:3}$\implies$\ref{car:2}:
Let $\psi:(A \rtimes_\alpha H,\ad(\lambda^\alpha)) \to (A_\infty,\alpha_\infty)$ be an equivariant approximate left-inverse for $\iota_A$. Since $\iota_A$ is non-degenerate, the image of $\psi$ is a non-degenerate \cstar-subalgebra of $D_{\infty,A}$. Hence, $\psi$ extends to a unital and equivariant $*$-homomorphism
\[
\psi:(\CM(A\rtimes_\alpha H),\ad(\lambda^\alpha))\to (\CM(D_{\infty,A}),\tilde{\alpha}_\infty).
\]
The unitaries $w_h = \psi(\lambda^\alpha_h)\in \CM(D_{\infty,A})$ define a unitary representation of $H$ and satisfy
\[
\alpha_h(a) = w_h a w_h^*\quad \text{and}\quad \tilde{\alpha}_{\infty,h}(w_g) = w_{hgh^{-1}}
\]
for all $a \in A$ and $g,h \in H$. 
\end{proof}

We continue with a duality result for actions of second-countable, compact, abelian groups with the Rokhlin property and approximately representable actions of countable, discrete, abelian groups on separable \cstar-algebras. This generalizes the well-known duality result by Izumi \cite[3.8]{Izumi2004} in the case of finite abelian group actions on separable, unital \cstar-algebras. Note that Gardella \cite{Gardella2014} has observed a similar phenomonon for circle actions on unital \cstar-algebras; see also \cite{GardellaPHD} for a further generalization in the unital case. The essential ingredients of the proof will be the characterizations \ref{seq split Rp} and \ref{char appr rep} of the Rokhlin property and approximate representability in terms of sequentially split $*$-homomorphisms. The duality result will turn out to be an application of the general duality principle \ref{seq split duality}. 
Before turning to the proof, some further preparation is needed. The following result is well-known.

\begin{prop}[cf.~{\cite[3.1]{Williams2007}}] \label{isomorphism group dual group}
Let $G$ be a locally compact, abelian group. The Gelfand transform yields an equivariant isomorphism
\[
(\mathrm{C}^*(\hat{G}),\kappa)\stackrel{\cong}{\longrightarrow}(C_0(G),\sigma),
\]
where $\kappa_g(\lambda^{\hat{G}}_\chi)=\chi(g^{-1})\lambda^{\hat{G}}_\chi$ and $\sigma_g(f)=f(g^{-1}\cdot\_)$ for all $g\in G$ and $\chi \in \hat{G}$.
\end{prop}

The following is taken from the proof of the Takai duality theorem \cite{Takai1975} presented in \cite[7.1]{Williams2007}, which is a variant of Raeburn's proof \cite{Raeburn1988}.

\begin{prop} \label{isomorphism proof Takai}
Let $G$ be a compact, abelian group, $A$ a \cstar-algebra and $\alpha:G\curvearrowright A$ a continuous action. There exists an equivariant isomorphism
\[
\psi:(A\rtimes_\alpha G\rtimes_{\hat{\alpha}}\hat{G},\ad(\lambda^{\hat{\alpha}})) \stackrel{\cong}{\longrightarrow} ((\CC(G)\otimes A)\rtimes_{\sigma\otimes \alpha} G, \widehat{\sigma\otimes\alpha})
\]
making the following diagram commute
\[
\xymatrix{
 (A\rtimes_\alpha G,\hat{\alpha}) \ar[r]^/-0.5cm/{\iota_{A\rtimes_\alpha G}} \ar[d]_{(\eins \otimes \id_A)\rtimes G} & (A\rtimes_\alpha G\rtimes_{\hat{\alpha}} \hat{G},\ad(\lambda^{\hat{\alpha}})) \ar[dl]^\psi \\
  ((\CC(G)\otimes A)\rtimes_{\sigma\otimes \alpha}G,\widehat{\sigma\otimes\alpha})
}
\]
\end{prop}
\begin{proof}
The $*$-homomorphism
\[
(\eins\otimes \iota_A)\rtimes G:A\rtimes_\alpha G\to (C^*(\hat{G})\otimes A )\rtimes_{\kappa\otimes \alpha} G
\]
and the unitary representation
\[
\hat{G}\to \CU(\CM((C^*(\hat{G})\otimes A)\rtimes_{\kappa\otimes \alpha}G)),\ \chi\mapsto  \lambda_\chi^{\hat{G}}\otimes \eins
\]
define a covariant pair for $(A\rtimes_\alpha G,\hat{\alpha})$. To see this, observe that $\mathrm{C}^*(\hat{G}) \otimes A$ is in the fixed point algebra of $\ad(\lambda^{\hat{G}}\otimes \eins)$ and that for $\chi\in \hat{G}$ and $g\in G$,
\[
\begin{array}{ccl}
(\lambda_\chi^{\hat{G}}\otimes\eins) \lambda_g^{\kappa\otimes \alpha} (\lambda_{\chi^{-1}}^{\hat{G}}\otimes \eins) & = & (\lambda_\chi^{\hat{G}} \kappa_g(\lambda_{\chi^{-1}}^{\hat{G}})\otimes \eins)\lambda_g^{\kappa\otimes \alpha}\\
 & = & \chi(g^{-1})^{-1}\lambda_g^{\kappa\otimes \alpha}\\
 & = & \chi(g)\lambda_g^{\kappa\otimes \alpha}.
\end{array}
\]
This also shows that $\widehat{\kappa \otimes \alpha} = \ad(\lambda^{\hat{G}}\otimes \eins)$. Hence, the corresponding integrated form $\psi': A\rtimes_\alpha G \rtimes_{\hat{\alpha}}\hat{G} \to (\mathrm{C}^*(\hat{G})\otimes A)\rtimes_{\kappa\otimes \alpha}G$ is $\ad(\lambda^{\hat{\alpha}})$-to-$\widehat{\kappa\otimes\alpha}$-equivariant. Therefore, the following diagram of equivariant $*$-homomorphisms commutes:
\[
\xymatrix{
(A\rtimes_\alpha G,\hat{\alpha}) \ar[r]^/-0.5cm/{\iota_{A\rtimes_\alpha G}} \ar[d]_{(\eins\otimes\id_A)\rtimes G} & (A\rtimes_\alpha G\rtimes_{\hat{\alpha}} \hat{G},\ad(\lambda^{\hat{\alpha}})) \ar[dl]^{\psi'} \\
 ((\mathrm{C}^*(\hat{G})\otimes A)\rtimes_{\kappa\otimes \alpha}G,\widehat{\kappa\otimes \alpha})
}
\]
To see that $\psi'$ is an isomorphism, note that the non-degenerate $*$-homomorphism
\[
\phi:\mathrm{C}^*(\hat{G})\otimes A \cong A\rtimes_{\id} \hat{G} \to A\rtimes_\alpha G\rtimes_{\hat{\alpha}}\hat{G},
\]
and the representation
\[
G\to \CU(\CM(A\rtimes_\alpha G\rtimes_{\hat{\alpha}}\hat{G})),\ g\mapsto \lambda_g^\alpha,
\]
define a covariant pair for $\kappa \otimes \alpha$. The covariance condition is indeed satisfied as
\[
\lambda_g^\alpha (a\lambda_\chi^{\hat{\alpha}})\lambda_{g^{-1}}^\alpha  =  \lambda_g^\alpha a\lambda^\alpha_{g^{-1}} \chi(g^{-1})\lambda_\chi^{\hat{\alpha}} = \alpha_g(a)(\chi(g^{-1})\lambda_\chi^{\hat{\alpha}}) = \phi((\kappa\otimes\alpha)_g(\lambda_\chi^{\hat{G}}\otimes a)).
\]
One checks that the resulting integrated form is inverse to $\psi'$. The claim now follows from \ref{isomorphism group dual group} because $(\mathrm{C}^*(\hat{G}),\kappa)$ may be replaced by $(\CC(G),\sigma)$.
\end{proof}

\begin{prop} \label{auxiliary iso}
Let $H$ be a discrete, abelian group, $A$ a \cstar-algebra and $\beta: H \curvearrowright A$ an  action. There exists an equivariant $*$-isomorphism
\[
\psi:(A \rtimes_\beta H \rtimes_{\ad(\lambda^\beta)} H, \widehat{\ad(\lambda^\beta)}) \stackrel{\cong}{\longrightarrow} (\CC(\hat{H}) \otimes (A \rtimes_\beta H), \sigma \otimes \hat{\beta})
\]
making the following diagram commute
\[
\xymatrix{
 (A\rtimes_\beta H,\hat{\beta}) \ar[r]^/-0.5cm/{\iota_A \rtimes H} \ar[d]_{\eins \otimes \id_{A \rtimes_\beta H}} & (A \rtimes_\beta H \rtimes_{\ad(\lambda^\beta)} H, \widehat{\ad(\lambda^\beta)}) \ar[dl]^\psi \\
  (\CC(\hat{H}) \otimes (A \rtimes_\beta H), \sigma \otimes \hat{\beta})
}
\]
\end{prop}
\begin{proof}
As $H$ is abelian, one checks that
\[
\iota_A \rtimes H : A\rtimes_\beta H \to A \rtimes_\beta H \rtimes_{\ad(\lambda^\beta)} H
\]
and the unitary representation
\[
\lambda^\beta : H \to \CU(\CM(A \rtimes_\beta H \rtimes_{\ad(\lambda^\beta)} H))
\]
define a covariant pair for $(A\rtimes_\beta H,\ad(\lambda^\beta))$. Note that we form the $*$-homomorphism $\iota_A\rtimes H$ by crossing with the second copy of $H$ in $\CM( A \rtimes_\beta H \rtimes_{\ad(\lambda^\beta)} H )$ and $\lambda^\beta$ maps into the first copy.
The corresponding integrated form is an automorphism $\psi_1 \in \Aut(A\rtimes_\beta H\rtimes_{\ad(\lambda^\beta)} H)$ making the following diagram of equivariant $*$-homomorphisms commute:
\[
\xymatrix{
 (A\rtimes_\beta H,\hat{\beta}) \ar[r]^/-0.5cm/{\iota_A\rtimes H} \ar[d]_{\iota_{A\rtimes_\beta H}} &  (A\rtimes_\beta H \rtimes_{\ad(\lambda^\beta)} H, \widehat{\ad(\lambda^\beta)}) \ar[dl]^{\psi_1} \\
 (A\rtimes_\beta H \rtimes_{\ad(\lambda^\beta)} H,\hat{\beta}\rtimes H)
}
\]
In short, $\psi_1$ flips the two copies of $H$ in $\CM( A \rtimes_\beta H \rtimes_{\ad(\lambda^\beta)} H )$.
Here, the expression $\hat{\beta}\rtimes H$ denotes the $\hat{H}$-action given by $(\hat{\beta}\rtimes H)_\chi = \hat{\beta}_\chi\rtimes H$ for every $\chi\in\hat{H}$. This is well-defined because for every $h \in H$ and $\chi\in\hat{H}$, we have $\ad(\lambda^\beta_h)\circ\hat{\beta}_\chi = \hat{\beta}_\chi\circ\ad(\lambda^\beta_h)$.

Since $\ad(\lambda^\beta)$ is unitarily implemeted, the covariant pair given by
\[
\eins\otimes\iota_{A\rtimes_\beta H}: A\rtimes_\beta H \into \mathrm{C}^*(H)\otimes (A\rtimes_\beta H)
\]
and
\[
H \to \CU(\CM(\mathrm{C}^*(H)\otimes (A\rtimes_\beta H))),\ h \mapsto \lambda^H_h\otimes\lambda_h^\beta
\]
gives rise to an isomorphism 
\[
\tilde{\psi}_2: A\rtimes_\beta H \rtimes_{\ad(\lambda^\beta)} H \stackrel{\cong}{\longrightarrow} \mathrm{C}^*(H)\otimes (A\rtimes_\beta H).
\]
Moreover, $\tilde{\psi}_2$ fits into the following diagram of equivariant $*$-homomorphisms:
\[
\xymatrix{
 (A\rtimes_\beta H, \hat{\beta}) \ar[r]^/-0.5cm/{\iota_{A \rtimes_\beta H}} \ar[d]_{\eins\otimes \id_{A \rtimes_\beta H}} & (A\rtimes_\beta H \rtimes_{\ad(\lambda^\beta)} H,\hat{\beta}\rtimes H) \ar[dl]^{\tilde{\psi}_2} \\
(\mathrm{C}^*(H)\otimes(A\rtimes_\beta H),\tilde{\kappa}\otimes \hat{\beta})
}
\]
where $\tilde{\kappa}_\chi(\lambda_h^H) = \chi(h)^{-1} \lambda_h^H$ for $\chi \in \hat{H}$ and $h\in H$. The map $\tilde{\psi}_2$ is equivariant because for $a \in A$, $\chi \in \hat{H}$ and $g,h \in H$, we have
\[
\begin{array}{lcl}
\tilde{\psi}_2 \circ (\hat{\beta}\rtimes H)_\chi(a\lambda_g^\beta \lambda_h^{\ad(\lambda^\beta)}) & = & \tilde{\psi}_2 (a\chi(g)\lambda_g^\beta\lambda_h^{\ad(\lambda^\beta)}) \\
& = & \chi(g)(\eins \otimes a\lambda_g^\beta) (\lambda_h^H \otimes \lambda_h^\beta)\\
& = & (\tilde{\kappa} \otimes \hat{\beta})_\chi \big( (\eins \otimes a\lambda_g^\beta)(\lambda_h^H \otimes \lambda_h^\beta) \big) \\
& = & (\tilde{\kappa} \otimes \hat{\beta})_\chi\circ \tilde{\psi}_2(a\lambda_g^\beta \lambda_h^{\ad(\lambda^\beta)})
\end{array}
\]
By the Pontryagin duality theorem and \ref{isomorphism group dual group}, there are equivariant $*$-isomorphisms $(\mathrm{C}^*(H),\tilde{\kappa}) \cong (\mathrm{C}^*(\hat{\hat{H}}),\kappa) \cong (\CC(\hat{H}),\sigma)$. Hence, $\tilde{\psi}_2$ gives rise to an equivariant $*$-isomorphism $\psi_2:(A\rtimes_\beta G \rtimes_{\ad(\lambda^\beta)} H, \hat{\beta}\rtimes H) \stackrel{\cong}{\longrightarrow} (\CC(\hat{H})\otimes(A\rtimes_\beta H),\sigma \otimes \hat{\beta})$. Now, $\psi = \psi_2 \circ \psi_1$ is the desired isomorphism.
\end{proof}

The following generalizes Izumi's duality result \cite[3.8]{Izumi2004} for finite abelian groups; see also \cite[4.4]{Nawata16}, \cite[3.6, 3.7]{Gardella2014} and \cite{GardellaPHD} for similar results, which are essentially contained in the next theorem.

\begin{theorem} \label{duality classical case}
Let $G$ be a second countable, compact, abelian group, $H$ a countable, discrete, abelian group and $A$ a separable \cstar-algebra. Let $\alpha: G\curvearrowright A$ be a continuous action and $\beta:H\curvearrowright A$ an action. Then 
\begin{enumerate}[label=\textup{(\roman*)},leftmargin=*]
\item $\alpha$ has the Rokhlin property if and only if $\hat{\alpha}: \hat{G}\curvearrowright A\rtimes_\alpha G$ is approximately representable; \label{dcc:1}
\item $\beta$ is approximately representable if and only if $\hat{\beta}:\hat{H}\curvearrowright A\rtimes_\beta H$ has the Rokhlin property. \label{dcc:2}
\end{enumerate}
\end{theorem}
\begin{proof}
The Pontryagin dual group $\hat{G}$ of a second countable, compact, abelian group $G$ is in fact countable, discrete, abelian, and vice versa. So the above statements make sense.
 
Let $\alpha:G \curvearrowright A$ be a continuous action. By \ref{seq split duality}, we know that
\[
\eins\otimes \id_A: (A,\alpha)\into (\CC(G)\otimes A,\sigma \otimes \alpha)
\]
is $G$-equivariantly sequentially split if and only if the induced $*$-homomorphism between the crossed products
\[
(\eins\otimes \id_A)\rtimes G:(A\rtimes_\alpha G,\hat{\alpha})\into ((\CC(G)\otimes A)\rtimes_{\sigma\otimes \alpha} G,\widehat{\sigma\otimes \alpha})
\]
is $\hat{G}$-equivariantly sequentially split. By applying \ref{isomorphism proof Takai}, we obtain a commutative diagram
\[
\xymatrix{
 (A\rtimes_\alpha G,\hat{\alpha}) \ar[r]^/-0.5cm/{\iota_{A\rtimes_\alpha G}} \ar[d]_{(\eins \otimes \id_A)\rtimes G} & (A\rtimes_\alpha G\rtimes_{\hat{\alpha}} \hat{G},\ad(\lambda^{\hat{\alpha}})) \ar[dl]^\cong \\
 ((\CC(G)\otimes A)\rtimes_{\sigma\otimes \alpha}G,\widehat{\sigma\otimes\alpha})
}
\]
We conclude that $\eins\otimes \id_A$ is $G$-equivariantly sequentially split if and only if $\iota_{A\rtimes_\alpha G}$ is $\hat{G}$-equivariantly sequentially split. It now follows from \ref{seq split Rp} and \ref{char appr rep} that $\alpha$ has the Rokhlin property if and only if $\hat{\alpha}$ is approximately representable. This shows \ref{dcc:1}.

For \ref{dcc:2}, let $\beta:H\curvearrowright A$ be an action. Again by \ref{seq split duality},
\[
\iota_A:(A,\beta)\into (A\rtimes_\beta H, \ad(\lambda^\beta))
\]
is $H$-equivariantly sequentially split if and only if the induced $*$-homomorphism
\[
\iota_A\rtimes H:(A\rtimes_\beta H,\hat{\beta})\into (A\rtimes_\beta H \rtimes_{\ad(\lambda^\beta)} H,\widehat{\ad(\lambda^\beta)})
\]
is $\hat{H}$-equivariantly sequentially split. By \ref{auxiliary iso}, we obtain a commutative diagram
\[
\xymatrix{
 (A\rtimes_\beta H,\hat{\beta}) \ar[r]^/-0.5cm/{\iota_A \rtimes H} \ar[d]_{\eins \otimes \id_{A \rtimes_\beta H}} & (A \rtimes_\beta H \rtimes_{\ad(\lambda^\beta)} H, \widehat{\ad(\lambda^\beta)}) \ar[dl]^\cong \\
  (\CC(\hat{H}) \otimes (A \rtimes_\beta H), \sigma \otimes \hat{\beta})}
\]
We conclude that $\iota_A$ is $H$-equivariantly sequentially split if and only if $\eins \otimes \id_{A \rtimes_\beta H}$ is $\hat{H}$-equivariantly sequentially split. It now follows from \ref{seq split Rp} and \ref{char appr rep} that $\beta$ is approximately representable if and only if $\hat{\beta}$ has the Rokhlin property. This shows \ref{dcc:2} and and finishes the proof.
\end{proof}


\subsection{Strongly self-absorbing actions}
\text{}

\noindent
Let us briefly recall from \cite{Szabo15ssa} the definition of a strongly self-absorbing action:

\begin{defi}
Let $\CD$ be a separable, unital \cstar-algebra and $G$ a second-countable, locally compact group. Let $\gamma: G\curvearrowright\CD$ be a continuous action. We say that $(\CD,\gamma)$ is a strongly self-absorbing \cstar-dynamical system, or that $\gamma$ is a strongly self-absorbing action, if the equivariant first-factor embedding
\[
\id_\CD\otimes\eins_\CD: (\CD,\gamma)\to (\CD\otimes\CD,\gamma\otimes\gamma)
\]
is approximately $G$-unitarily equivalent to an isomorphism $\mu$, that is, there exists a sequence $u_n \in \CU(\CD \otimes \CD)$ such that $\mu = \lim_{n \to \infty} \ad(u_n)\circ(\id_\CD\otimes\eins_\CD )$ and $\lim_{n \to \infty} \|(\gamma \otimes \gamma)_g (u_n) - u_n\| = 0$ uniformly on compact subsets of $G$.
\end{defi}

Combining the main result of \cite{Szabo15ssa} with \ref{eq unital homs FA}, we see that equivariant tensorial absorption of a strongly self-absorbing action can also be expressed in the language of sequentially split $*$-homomorphisms:

\begin{theorem}[cf.~{\cite[3.7]{Szabo15ssa}}]
\label{equi-McDuff}
Let $G$ be a second-countable, locally compact group. Let $A$ be a separable \cstar-algebra and $\alpha: G\curvearrowright A$ a continuous action. Let $\CD$ be a separable, unital \cstar-algebra and $\gamma: G\curvearrowright\CD$ a continuous action such that $(\CD,\gamma)$ is strongly self-absorbing. The following are equivalent:
\begin{enumerate}[label=\textup{(\roman*)},leftmargin=*] 
\item $(A,\alpha)$ is strongly cocycle conjugate to $(A\otimes\CD,\alpha\otimes\gamma)$. 
\item $(A,\alpha)$ is cocycle conjugate to $(A\otimes\CD,\alpha\otimes\gamma)$.
\item There exists a unital, equivariant $*$-homomorphism from $(\CD,\gamma)$ to $\big( F_{\infty,\alpha}(A), \tilde{\alpha}_\infty \big)$. 
\item The first-factor embedding $\id_A\otimes\eins_\CD: (A,\alpha)\to (A\otimes\CD,\alpha\otimes\gamma)$ is equivariantly sequentially split.
\end{enumerate}
\end{theorem}

With the help of some observations from the third section, we deduce the following interesting consequences with the help of the above perspective:

\begin{theorem}
Let $G$ be a second-countable, locally compact group. Let $A$ be a separable \cstar-algebra and $\alpha: G\curvearrowright A$ a continuous action. Let $\CD$ be a separable, unital \cstar-algebra and $\gamma: G\curvearrowright\CD$ a strongly self-absorbing action. Assume that $(A,\alpha)$ is (strongly) cocycle conjugate to $(A\otimes\CD, \alpha\otimes\gamma)$.
\begin{enumerate}[label=\textup{(\roman*)},leftmargin=*]
\item If $E\subset A$ is a hereditary and $\alpha$-invariant \cstar-subalgebra, then $(E,\alpha|_E)$ is (strongly) cocycle conjugate to $(E\otimes\CD, \alpha|_E\otimes\gamma)$.
\item If $\beta: G\curvearrowright B$ is another continuous action on a separable \cstar-algebra such that $(A,\alpha)$ and $(B,\beta)$ are equivariantly Morita equivalent, then $(B,\beta)$ is (strongly) cocycle conjugate to $(B\otimes\CD, \beta\otimes\gamma)$.
\end{enumerate}
\end{theorem}
\begin{proof}
This follows directly from \ref{equi-McDuff}, \ref{eq-seq-split-hered} and \ref{seq split stability}.
\end{proof}

\begin{reme}
One can also obtain the two theorems above for semi-strongly self-absorbing actions $\gamma$ (see \cite[Section 4]{Szabo15ssa}) with the identical argument.
\end{reme}


\bibliographystyle{gabor}
\bibliography{sbarlak}

\end{document}